\documentclass{amsart}

\usepackage{amsmath,amsfonts,amsthm,amssymb,graphics,graphicx, verbatim,xcolor}
\usepackage{esint}
\usepackage[margin=1.6cm]{geometry}
\usepackage{graphicx,url}
\graphicspath{{./Images/}}
\usepackage{float}
\usepackage[export]{adjustbox}
\newtheorem{thm}{Theorem}[section]

\newtheorem{cor}[thm]{Corollary}
\newtheorem{lem}[thm]{Lemma}

\newtheorem{prop}[thm]{Proposition}
\theoremstyle{remark}

\newcommand{\norm}[1]{\left|\!\left|{#1}\right|\!\right|}
\newcommand{\Vol}{\ensuremath{\text{Vol}}}
\renewcommand{\S}{\ensuremath{\mathbb{S}}}

\newcommand{\R}{\ensuremath{\mathbb{R}}}

\newcommand{\E}{\ensuremath{\mathbb{E}}}
\newcommand{\N}{\ensuremath{\mathbb{N}}}

\begin{document}
\title{Small scale equidistribution of random waves generated by an unfair coin flip}
\author{Miriam Leonhardt}
\address{Department of Mathematics, University of Auckland, New Zealand }
\email{mleo705@aucklanduni.ac.nz}

\author{Melissa Tacy}
\address{Department of Mathematics, University of Auckland, New Zealand }
\email{melissa.tacy@auckland.ac.nz}


\thanks{The authors would like to acknowledge Jeroen Schillewaert for asking Tacy what would happen to random wave equidistribution if the coefficients were not from a fair coin type distribution. Further we thank our reviewer for their comments and advice which improved the paper. The authors also acknowledge the support of the University of Auckland's Summer Scholar Scheme which funded this project.}

\begin{abstract}
In this paper we study the small scale equidistribution property of random waves whose coefficients are determined by an unfair coin. That is the coefficients take value $+1$ with probability $p$ and $-1$ with probability $1-p$.  Random waves whose coefficients are associated with a fair coin are known to equidistribute down to the wavelength scale. We obtain explicit requirements on the deviation from the fair ($p=0.5$) coin to retain equidistribution. \end{abstract}

\maketitle
Lately there been a renewed interest in the properties of random waves, in particular their small scale equidistribution properties. Berry \cite{Ber77} introduced ensembles of random waves as a model for chaotic billiards. Random waves are functions of $\R^{n}$ of the form
\begin{equation}\sum_{\xi_{j}\in \Lambda}C_{j}e^{i\lambda  x\cdot\xi_{j}}\label{rw}\end{equation}
where the coefficients $C_{j}$ are chosen according to a some probability distribution and $\Lambda\subset \S^{n-1}$. Common choices of coefficients include independent random variables such as Gaussian or Rademacher random variables (see for instance \cite{Ber77},\cite{Zel09} and \cite{dCI20}) and uniform probability density on high dimensional unit spheres (see for instance \cite{Zel96},\cite{BL13},\cite{Map13},\cite{Zel14} and \cite{H17}). Usually $\Lambda$ is chosen so that the directions $\xi_{j}$ are equally spaced with spacing less than one wavelength, $\lambda^{-1}$. 

The property of equidistribution (in configuration space) is that the $L^{2}$ density of $u$ is equally spread throughout the domain. Since random waves are defined on an infinite domain typically studies on random waves restrict attention to the ball of radius one about zero and normalise so that
$$\E\left[\int_{B_{1}(0)}|u(x)|^{2}dx\right]=\Vol(B_{1}(0)).$$
In the sense of Berry's model we should understand random waves as representing the behaviour of quantum states in chaotic systems. Therefore by restricting to the ball of radius one about zero we are defining this space to act as our ``universe'' and the normalisation convention tells us that (at least in expectation) the state lives in the universe with probability one. In the context of this normalisation we say that a random wave is strongly equidistributed on a set $X\subset B_{1}(0)$ if
\begin{equation}
\E\left[\int_{X}|u(x)|^{2}dx\right]=\Vol(X)\left(1+o(1)\right)\label{equiXs}\end{equation}
and
\begin{equation}
\sigma^{2}\left[\int_{X}|u(x)|^{2}dx\right]=o((\Vol(X))^{2}).\label{varXest}\end{equation}
In this paper we also allow for a concept of weak equidistribution where \eqref{varXest} holds but \eqref{equiXs} is replaced by
\begin{equation}
c\Vol(X)\leq \E\left[\int_{X}|u(x)|^{2}dx\right]\leq C\Vol(X).\label{equiXw}\end{equation}
So in the setting of weak equidistribution the probability of a state being located in the set $X$ is proportional to the volume of $X$.

In this paper we are interested in the two dimensional problem where $X$ is a small ball (one whose radius decays to zero as some power of $\lambda^{-1}$). For convenience we will consider the ball about the origin however none of our analysis is dependent on this centre point so the results hold for balls centred around general points $p\in \R^{2}$. In the setting of manifolds the question of equidistribution on small balls where the coefficients are uniformly distributed on the sphere or Gaussian are resolved in \cite{HanTac20} and \cite{dCI20} respectively. While Rademacher coefficients have not been explicitly studied, most of the results of \cite{dCI20} rely on properties of Gaussian random variables that are shared by Rademacher coefficients. The conclusion of these papers is that strong equidistribution of random waves holds on small balls of radius $\lambda^{-\alpha}$ so long as $\alpha<1$. Here we consider a variant of the Rademacher $\pm 1$ coefficients, one associated with an ``unfair coin''. That is we assign each coefficient the value $+1$ with probability $p$ and $-1$ with probability $1-p$. As with Rademacher and Gaussian coefficients the individual coefficients remain independent of each other. We ask, just how unfair does the coin have to be before we lose the property of equidistribution?

Before stating the theorems of this paper it is worth considering how large $\int_{B_{r}(0)}|u(x)|^{2}dx$ can be if we  do not randomise coefficients. From Sogge \cite{Sogge16} we see that for eigenfunctions (and in fact spectral clusters) on Riemannian manifolds $(M,g)$,
\begin{equation}\norm{u}_{L^{2}(B_{r}(0))}\leq r^{1/2}\norm{u}_{L^{2}(M)}\label{Soggebnd}\end{equation}
and that in fact this upper bound has sharp examples. The same is true for approximate eigenfunctions on $\R^{2}$. Consider for example the function given by
$$v(x)=\lambda^{\frac{1}{2}}\int_{\S}e^{i\lambda x\cdot \xi}d\mu(\S).$$
That is the ($L^{2}$ normalised) inverse Fourier transform of the surface measure of the unit circle $\S$. This example has a significant history in the analysis of restriction operators and is the standard example for sharpness of the Fourier restriction problem when $p<\frac{2n}{n+1}$, (see for example section 1.2 of \cite{Tao04}). Using the method of stationary phase it can be shown that
$$|v(x)|=C\left(1+\lambda|x|\right)^{-\frac{1}{2}}$$
and therefore saturates \eqref{Soggebnd}. For comparison an equidistributed eigenfunction would have $\norm{u}_{L^{2}(B_{r}(0))}\approx r\norm{u}_{L^{2}(M)}$. 

Let us look at the extreme case of a completely unfair coin. In this case we always have a coefficient of $+1$. Then
\begin{equation}u=\sum_{\xi_{j}\in \Lambda}e^{i\lambda x\cdot \xi_{j}}.\label{unfairsum}\end{equation}
Supposing that the $\xi_{j}$ are spaced at scales much smaller than the wavelength we would then expect to be able to replace the sum in \eqref{unfairsum} with an integral (and indeed in Section \ref{sec:mainexpint} we perform just such a replacement). Then we have
$$u=C_{\Lambda}\int e^{i\lambda x\cdot \xi}d\mu(\S)+\text{Error}$$
where $C_{\Lambda}$ is a re-normalisation constant that depends on the number of element of $\Lambda$ and the error term is small enough to be ignored. Notice that in this case $u$ is (up to a constant and an error term) equal to the inverse Fourier transform of surface measure. Therefore in the extreme case of a completely unfair coin the growth of random waves on small balls is no better than that of eigenfunctions in general while those associated with a completely fair coin are equidistributed. 

We now address the intermediate cases. For the purposes of this paper rather than considering
$$\int_{B_{0}(r)}|u(x)|^{2}dx$$
for $r=\lambda^{-\alpha}$ we look at a smoothed version
$$\int a^{2}(\lambda^{\alpha}|x|)|u(x)|^{2}dx$$
where $a(r)$ is a smooth, cut-off function supported on $[-2,2]$ and assumed to be equal to one on $[-1,1]$. We first obtain upper bounds for 
$$\E[\norm{a_{\lambda}u}_{L^{2}}^{2}]=\E\left[\int a_{\lambda}(x)|u(x)|^{2}dx\right]=\E\left[\int a^{2}(\lambda^{\alpha}|x|)|u(x)|^{2}dx\right]$$
in the case where $\Lambda$ is a set of $N=\gamma\lambda$ equi-spaced directions $\xi_{j}$ with $(\lambda\gamma)^{-1}$ spacing.

\begin{thm}\label{thm:mainexpsum}
Suppose $u$ is a random wave given by \eqref{rw} where $\Lambda$ is set of equi-spaced directions $\xi_{j}$ with spacing $(\lambda\gamma)^{-1}$ and the coefficients are independent random variables each taking the value $+1$ with probability $p$ and $-1$ with probability $1-p$. Then, for $\alpha<1$,
\begin{equation}\E\left[\norm{a_{\lambda}u}_{L^{2}}^{2}\right]\leq C(\gamma\lambda^{1-2\alpha}+(2p-1)^{2}\gamma^{2}\lambda^{1-\alpha})\label{eqn:expest}.\end{equation}
\end{thm}

Ideally we would also like to obtain a lower bound (since this would allow us to explore weak equidistribution). To obtain the lower bound it is necessary to replace various sums with integrals, see Section \ref{sec:mainexpint}. This replacement should be understood as giving us lower bounds when the spacing between directions becomes significantly smaller than the wavelength associated with the oscillation. In our model this would correspond to making $\gamma$ large. 

\begin{thm}\label{thm:mainexpint}
Suppose $u$ is a random wave given by \eqref{rw} where $\Lambda$ is set of equi-spaced directions $\xi_{j}$ with spacing $(\lambda\gamma)^{-1}$ and the coefficients are independent random variables each taking the value $+1$ with probability $p$ and $-1$ with probability $1-p$. Then, for $\alpha<1$, 
\begin{equation} c(\gamma\lambda^{1-2\alpha}+(2p-1)^{2}\gamma^{2}\lambda^{1-\alpha})\leq \E\left[\norm{a_{\lambda}u}_{L^{2}}^{2}\right]\leq C(\gamma\lambda^{1-2\alpha}+(2p-1)^{2}\gamma^{2}\lambda^{1-\alpha})\label{eqn:explower}.\end{equation}
\end{thm}

The final ingredient in our understanding of equidistribution is control of the variance. Note that the expectation could be equidistributed by the values fluctuating wildly so that a ``typical'' random wave was in fact not equidistributed. This is indeed the case for the ``fair coin'' distribution on balls smaller than the wavelength, $r\leq \lambda^{-1}$. If however the variance decays in comparison to the (normalised) volume of the ball then typical random waves from this distribution will equidistribute. 

\begin{thm}\label{thm:mainvar}
Suppose $u$ is a random wave given by \eqref{rw} where $\Lambda$ is set of equi-spaced directions $\xi_{j}$ with spacing $(\lambda\gamma)^{-1}$ and the coefficients are independent random variables each taking the value $+1$ with probability $p$ and $-1$ with probability $1-p$. Then, for $\alpha<1$, 
\begin{equation}
\sigma^{2}\left[\norm{a_{\lambda}u}_{L^{2}}^{2}\right]\leq C_1\lambda^{1-3\alpha}\gamma^{2}\left(1-(2p-1)^{2}\right)^2+C_2\gamma^3\lambda^{1-2\alpha}(2p-1)^2(1-(2p-1)^2)\label{eqn:varest}.\end{equation}
\end{thm}
Now we can begin to answer the question of equidistribution. We will normalise so that $\E(\norm{a_{1}u}_{L^{2}})=1$ for a fair coin randomisation and compare our results to their normalised volume. As we will see from the expectation calculation in Section \ref{sec:expsum} this normalisation can be achieved by multiplying $u$ by a prefactor of $\gamma^{-1/2}\lambda^{-1/2}$. Recall that we are assuming $\gamma$ is large, we do not however want to take it too large (doing so reduces orthogonality relationships). Our interest is in balls so that $1/r$ grows as a power of $\lambda$. To that end we choose a softer growth rate for $\gamma$ and while we allow $\gamma\to\infty$ we assume that $\gamma\leq \log(\lambda)$. 

From Corollary \ref{cor:equi} we see that equidistribution is preserved if
\begin{equation}\notag
p=0.5+\mathcal{O}( \lambda^{-\frac{\alpha}{2}}\gamma^{-\frac{1}{2}})
\end{equation}
so if we also assume only a logarithmic type growth for $\gamma$ any probability of the form $p=0.5+\lambda^{-\beta}$, where $\beta>\alpha/2$, retains the correct expectation. Using Theorem \ref{thm:mainvar} (and normalising) we get that a normalised unfair random wave has variance bounded by $C_1\lambda^{-1-3\alpha}(1-(2p-1)^2)^{2}+C_2\gamma\lambda^{-1-2\alpha}(2p-1)^2(1-(2p-1)^2)$. Using the condition for equidistribution from Corollary \ref{cor:equi}, the second term is of the same size as the first term: 
$$\sigma^2\leq C_1\lambda^{-1-3\alpha}(1-(2p-1)^2)^{2}+C_2\lambda^{-1-3\alpha}(1-(2p-1)^2).$$ So as long as $\alpha<1$ the variance is sufficiently controlled for $p$ sufficiently close to 0.5. This is discussed in further detail in the lead up to Corollary \ref{cor:var}. 

This paper is arranged in the following fashion. First, in Section \ref{sec:expsum}, we obtain the upper bound of Theorem \ref{thm:mainexpsum}. Then in Section \ref{sec:mainexpint} we replace the sums appearing in our expression for expectation with integrals. We are then able to compute those integrals via the method of stationary phase to obtain Theorem \ref{thm:mainexpint}. Finally in Section \ref{sec:mainvar} we obtain the upper bounds on the variance given in Theorem \ref{thm:mainvar}.

In this paper we adopt the notation $f\lesssim g$ to mean that
$$f\leq Cg$$
where $C$ is a constant independent of the parameters $\lambda$ and $\gamma$ but may change from line to line.

\section{Proof of Theorem \ref{thm:mainexpsum}}\label{sec:expsum}

In this section we will obtain an upper bound on $\E(\norm{a_{\lambda}u}_{L^{2}}^{2})$ for any set of directions $\Lambda$ that are equally spaced on $\S$. Later we will use this and an approximation of sums by integrals to obtain more refined asymptotics.  First we write
 $$\E(\lVert a_\lambda u\rVert^2)=\sum_{k}P_k\int{\sum_{j,l}C^{(k)}_jC^{(k)}_la^{2}(\lambda^\alpha |x|)e^{i\lambda x\cdot\xi_j}e^{-i\lambda x\cdot\xi_l}dx}$$
where $P_k$ is the probability of a random $C$-vector (the vector which stores the values of the $C_j$) being $C^{(k)}$ and sums over $k$ represents the sum over all $C$-vectors. The sum $\sum\limits_kP_k=1$.\\

Since there is only a finite number of $j$ and $l$ ($N$ of each), there are $N^2$ pairs and there is a finite number $(2^N)$ of possible $C$-vectors. This means that both sums involved in the expectation value are finite, so they converge, and their order can be interchanged. Similarly, finite sums commute with integrals so their order can also be swapped, giving:

\begin{align}
\E(\lVert a_\lambda u\rVert^2)&=\sum_{j,l}\left(\sum_kP_kC^{(k)}_jC^{(k)}_l\int{a^{2}(\lambda^\alpha |x|)e^{i\lambda x\cdot(\xi_j-\xi_l)}dx}\right)\notag\\
&=\sum_{j}\left(\sum_{k}P_k\left(C^{(k)}_j\right)^2\int{a^{2}(\lambda^\alpha |x|)e^{i\lambda x\cdot(0)}dx}\right)+\sum_{\substack{j,l\\j\neq l}}\left(\sum_kP_kC^{(k)}_jC^{(k)}_l\int{a^{2}(\lambda^\alpha |x|)e^{i\lambda x\cdot(\xi_j-\xi_l)}dx}\right)\notag\\
&=\sum_j\sum_kP_k\int{a^{2}(\lambda^\alpha |x|)dx}+\sum_{\substack{j,l\\j\neq l}}\left(\sum_kP_kC^{(k)}_jC^{(k)}_l\int{a^{2}(\lambda^\alpha |x|)e^{i\lambda x\cdot(\xi_j-\xi_l)}dx}\right)\notag\\
&=N\int{a^{2}(\lambda^\alpha |x|)dx}+\sum_{\substack{j,l\\j\neq l}}\left(\sum_kP_kC^{(k)}_jC^{(k)}_l\int{a^{2}(\lambda^\alpha |x|)e^{i\lambda x\cdot(\xi_j-\xi_l)}dx}\right)\notag\\
&=N\norm{a^{2}_{\lambda}}_{L^{1}}+\sum_{\substack{j,l\\j\neq l}}\left(\sum_kP_kC^{(k)}_jC^{(k)}_l\int{a^{2}(\lambda^\alpha |x|)e^{i\lambda x\cdot(\xi_j-\xi_l)}dx}\right)\label{sumsE}.
\end{align}

From here on, we will refer to the first term in this expression (\eqref{sumsE}) as the diagonal term and to the second term as the off-diagonal terms. To make progress in the calculation we need to evaluate $\sum_kP_kC_j^{(k)}C_l^{(k)}$ in terms of $p$:

\begin{lem}\label{prob}
For each $j\neq l$ pair, where the coefficients $C_j$ and $C_l$ are independent random variables which take on the value $+1$ with a probability of $p$ or $-1$ with a probability of $1-p$: 
$$\sum_kP_kC^{(k)}_jC^{(k)}_l=(2p-1)^2.$$ 
\end{lem}

\begin{proof}
Since $\sum_kP_kC_j^{(k)}C_l^{(k)}=\E(C_j^{(k)}C_l^{(k)})$ and the values of entry $j$ and $l$ are independent of each other, $\E(C_j^{(k)}C_l^{(k)})=\E(C_j^{(k)})\cdot\E(C_l^{(k)})$. As the probability of $C_j^{(k)}$ being $+1$ is $p$ and the probability of $C_j^{(k)}$ being $-1$ is $1-p$, we have $\E(C_j^{(k)})=\sum_kP_kC_j^{(k)}=(+1)p+(-1)(1-p)=2p-1$. Therefore $\sum_kP_kC_j^{(k)}C_l^{(k)}=\E(C_j^{(k)}C_l^{(k)})=(2p-1)^2$.
\end{proof}

 Now that we have evaluated $\sum_kP_kC_j^{(k)}C_l^{(k)}$ in terms of $p$, we can substitute this into \eqref{sumsE} to give
\begin{align}
\E(\lVert a_\lambda u\rVert^2)&=N\norm{a^{2}_{\lambda}}_{L^{1}}+(2p-1)^2\sum_{\substack{j,l\\j\neq l}}\int{a^{2}(\lambda^\alpha |x|)e^{i\lambda x\cdot(\xi_j-\xi_l)}dx}\label{Expectation}\\
&=\gamma\lambda\norm{a^{2}_{\lambda}}_{L^{1}}+(2p-1)^2\sum_{\substack{j,l\\j\neq l}}\int{a^{2}(\lambda^\alpha |x|)e^{i\lambda x\cdot(\xi_j-\xi_l)}dx}\notag.
\end{align}

Note that since $a_{\lambda}$ is supported on the ball of radius $2\lambda^{-\alpha}$  we can say that $\norm{a^{2}_{\lambda}}_{L^{1}}\leq C\lambda^{-2\alpha}$ and arrive at
\begin{equation}
\E(\lVert a_\lambda u\rVert^2)\leq C\gamma\lambda^{1-2\alpha}+(2p-1)^2\sum_{\substack{j,l\\j\neq l}}\int{a^{2}(\lambda^\alpha |x|)e^{i\lambda x\cdot(\xi_j-\xi_l)}dx.}\label{OintE}\end{equation}

Therefore all that remains in order to obtain the upper bound is to estimate the integrals in the off-diagonal term. These  are oscillatory integrals. Oscillatory integrals are integrals which involve a highly oscillatory function, that alternates between positive and negative values, so that there is a high degree of cancellation. The frequency at which the function is oscillating determines how much cancellation there is, and for high frequencies we can often use the oscillation to get a decay in the size of the integral. For the specific integral in our expression \eqref{OintE}, we address this in the following theorem.

\begin{thm}\label{OscillatoryInt}
If $a(\lambda^\alpha |x|)$ is a smooth cutoff function with compact support on the ball of radius $r=2\lambda^{-\alpha}$ centred at 0, and $\xi_j-\xi_l\neq0$ then for all $n\in\N$:
$$\left\lvert\int{a^{2}(\lambda^\alpha |x|)e^{i\lambda x\cdot(\xi_j-\xi_l)}dx}\right\rvert\leq C_n\lambda^{-2\alpha}\left(\frac{\lambda^{(-1+\alpha)}}{|\xi_j-\xi_l|}\right)^n.$$ 
\end{thm}
\begin{proof}
In these oscillatory integrals $\phi(x)=x\cdot(\xi_j-\xi_l)$. Due the properties of the exponential function we can write:
$$\int{a^{2}(\lambda^\alpha |x|)e^{i\lambda x\cdot(\xi_j-\xi_l)}dx}=\int{\frac{1}{i\lambda\partial_v\phi(x)}a^{2}(\lambda^\alpha |x|) \partial_v\left(e^{i\lambda x\cdot(\xi_j-\xi_l)}\right)dx}$$
where $v$ is a normalised direction vector and $\partial_v$ is the directional derivative in the direction of $v$. Since $\nabla\phi(x)=\xi_j-\xi_l$ is constant and non-zero, it makes sense to pick $v=\frac{\nabla\phi(x)}{|\nabla\phi(x)|}$ since this will give the most effective upper bound as the directional derivative will take its maximum value of $|\nabla\phi(x)|=|\xi_j-\xi_l|$. Therefore we can use integration by parts in the direction of the gradient to transfer the derivative from the exponential to the function $a_\lambda$: 
\begin{equation}\label{basecase}
\int{a^{2}(\lambda^\alpha |x|)e^{i\lambda x\cdot(\xi_j-\xi_l)}dx}=\int{\frac{-1}{i\lambda|\xi_j-\xi_l|}e^{i\lambda x\cdot(\xi_j-\xi_l)}\lambda^\alpha\partial_v(a^{2}(\lambda^\alpha |x|))dx}.
\end{equation}
The boundary terms are zero due to the cutoff function $a^{2}(\lambda^\alpha x)$.
\eqref{basecase} is the base case for the inductive argument we use to show that:
\begin{equation}\label{induction}
\int{a^{2}(\lambda^\alpha |x|)e^{i\lambda x\cdot(\xi_j-\xi_l)}dx}=\int{\left(\frac{-1}{i\lambda|\xi_j-\xi_l|}\right)^ne^{i\lambda x\cdot(\xi_j-\xi_l)}\lambda^{n\alpha}\partial_v^{(n)}(a^{2}(\lambda^\alpha |x|))dx}\text{, }\forall n\in\N.
\end{equation}
To complete the inductive argument we must show that if it is true for $k$, it is true for $k+1$:
\begin{align}
\int{a^{2}(\lambda^\alpha |x|)e^{i\lambda x\cdot(\xi_j-\xi_l)}dx}&=\int{\left(\frac{-1}{i\lambda|\xi_j-\xi_l|}\right)^k\frac{1}{i\lambda|\xi_j-\xi_l|}\lambda^{k\alpha}\partial_v^{(k)}(a^{2}(\lambda^\alpha |x|))\cdot\partial_v(e^{i\lambda x\cdot(\xi_j-\xi_l)})dx}\notag\\
&=\int{-\left(\frac{-1}{i\lambda|\xi_j-\xi_l|}\right)^k\frac{1}{i\lambda|\xi_j-\xi_l|}e^{i\lambda x\cdot(\xi_j-\xi_l)}\cdot\lambda^{k\alpha}\cdot\lambda\cdot\partial_v^{(k+1)}(a^{2}(\lambda^\alpha |x|))dx}\notag\\
&=\int{\left(\frac{-1}{i\lambda|\xi_j-\xi_l|}\right)^{k+1}e^{i\lambda x\cdot(\xi_j-\xi_l)}\lambda^{(k+1)\alpha}\partial_v^{(k+1)}(a^{2}(\lambda^\alpha |x|))dx}.\notag
\end{align}
Therefore by the principle of mathematical induction \eqref{induction} is true $\forall n\in\N$. This means that 
\begin{align}
\left\lvert\int{a^{2}(\lambda^\alpha|x|)e^{i\lambda x\cdot(\xi_j-\xi_l)}dx}\right\rvert&=\left\lvert\int{\left(\frac{-1}{i\lambda|\xi_j-\xi_l|}\right)^ne^{i\lambda x\cdot(\xi_j-\xi_l)}\lambda^{n\alpha}\partial_v^{(n)}(a^{2}(\lambda^\alpha |x|))dx}\right\rvert\notag\\
&\leq\int{\left\lvert\left(\frac{-1}{i\lambda|\xi_j-\xi_l|}\right)^ne^{i\lambda x\cdot(\xi_j-\xi_l)}\lambda^{n\alpha}\partial_v^{(n)}(a^{2}(\lambda^\alpha |x|))\right\rvert dx}\notag\\
&=\int{\frac{\lambda^{n\alpha}}{\lambda^n|\xi_j-\xi_l|^n}|\partial_v^{(n)}(a^{2}(\lambda^\alpha |x|))|dx}.\notag
\end{align}
Since the cutoff function has compact support, on the ball of radius $r=2\lambda^{-\alpha}$, and since the integrand is positive we can get an upper bound by taking the region of integration to be $|x|\leq2\lambda^{-\alpha}$, and by letting $C_n$ be a positive constant that bounds the derivatives of $a^2_\lambda$;
$$\left\lvert\int{a^{2}(\lambda^\alpha|x|)e^{i\lambda x\cdot(\xi_j-\xi_l)}dx}\right\rvert\leq\frac{\lambda^{n(-1+\alpha)}}{|\xi_j-\xi_l|^n}\int_{|x|\leq2\lambda^{-\alpha}}{C_ndx}=C_n\left(\frac{\lambda^{-1+\alpha}}{|\xi_j-\xi_l|}\right)^n\int_{|x|\leq2\lambda^{-\alpha}}{1dx}.$$
The volume of this region is $4\pi\lambda^{-2\alpha}$, where the constants which are independent of $\lambda$ can be absorbed into $C_n$:
$$\left\lvert\int{a^2(\lambda^\alpha |x|)e^{i\lambda x\cdot(\xi_j-\xi_l)}dx}\right\rvert\leq C_n\lambda^{-2\alpha}\left(\frac{\lambda^{-1+\alpha}}{|\xi_j-\xi_l|}\right)^n.$$
\end{proof}
From this we can see that the absolute value of the integral decays with $\lambda^{-1+\alpha}$ and that for high frequencies, corresponding to large values of $\lambda$, this means the integral has a small value.
This upper bound is only effective if the factor that appears with each integration by parts is less than one, otherwise it increases the value each time, i.e. $\frac{\lambda^{-1+\alpha}}{|\xi_j-\xi_l|}<1$. 
If this is not the case, i.e. $\lambda^{-1+\alpha}\geq|\xi_j-\xi_l|$, then the oscillations are occurring at a low frequency, since the smallness of $|\xi_j-\xi_l|$ counteracts the rapid oscillations due to large values of $\lambda$. This means that there will not be much cancellation due to oscillations so one can obtain an effective upper bound using:
\begin{multline}\notag
\left\lvert\int{a^2(\lambda^\alpha |x|)e^{i\lambda x\cdot(\xi_j-\xi_l)}dx}\right\rvert\leq\int{\left\lvert a^2(\lambda^\alpha |x|)e^{i\lambda x\cdot(\xi_j-\xi_l)}\right\rvert dx}=\int{|a^2(\lambda^\alpha |x|)|dx}\\ 
\leq \int_{|x|\leq2\lambda^{-\alpha}}{C_ndx}\leq C_nVol(B_{2\lambda^{-\alpha}}(0))\leq C_n\lambda^{-2\alpha}
\end{multline}
where the constants which don't depend on $\lambda$ have been absorbed into $C_n$ (bounds on the derivatives of $a^2_\lambda$).\\
Both these cases can be combined and written as the following equation, which holds $\forall n\in\N$: 
\begin{equation}\label{combined}
\left\lvert\int{a^2(\lambda^\alpha |x|)e^{i\lambda x\cdot(\xi_j-\xi_l)}dx}\right\rvert\leq C_n\lambda^{-2\alpha}\left(1+\frac{|\xi_j-\xi_l|}{\lambda^{-1+\alpha}}\right)^{-n}.
\end{equation}
This works in the case where $\frac{\lambda^{-1+\alpha}}{|\xi_j-\xi_l|}\leq1$ since in this case $\frac{|\xi_j-\xi_l|}{\lambda^{-1+\alpha}}\geq1$ meaning that this term is the dominant term in the expression, \eqref{combined}, for the bound, and the 1 can be ignored, giving the same bound as before.  Similarly, if  $\frac{\lambda^{-1+\alpha}}{|\xi_j-\xi_l|}>1$, one has that $\frac{|\xi_j-\xi_l|}{\lambda^{-1+\alpha}}<1$, meaning the 1 is the dominant term in \eqref{combined}, and the other term can be ignored. This also gives the correct bound for the second case. \\
\\
For a fixed, finite, positive integer $n$, which is sufficiently large to cancel out the decay in $\lambda$, one can pick $C=\max\{C_m|m\leq n\}$ since $a^2_\lambda$ is smooth, so its derivatives are all bounded. Then
\begin{equation}\label{bothcases}
\int{a^2(\lambda^\alpha |x|)e^{i\lambda x\cdot(\xi_j-\xi_l)}dx}\leq C_n\lambda^{-2\alpha}\left(1+\frac{|\xi_j-\xi_l|}{\lambda^{-1+\alpha}}\right)^{-n}\leq C\lambda^{-2\alpha}\left(1+\frac{|\xi_j-\xi_l|}{\lambda^{-1+\alpha}}\right)^{-n}.
\end{equation}
To be able to find an upper bound for the expectation, \eqref{OintE}, we need to find an upper bound for the double sum $\sum_{\substack{j,l\\j\neq l}}\int{a^2(\lambda^\alpha |x|)e^{i\lambda x\cdot(\xi_j-\xi_l)}dx}$. By fixing a value of $j$, the bounds determined above, \eqref{bothcases} can be used to find an upper bound for the sum: $\sum_{\substack{l\\l\neq j}}\int{a(\lambda^\alpha |x|)e^{i\lambda x\cdot(\xi_j-\xi_l)}dx}$.
\begin{equation}\label{predyadic}
\sum_{\substack{l\\l\neq j}}\int{a^2(\lambda^\alpha |x|)e^{i\lambda x\cdot(\xi_j-\xi_l)}dx}\leq C\lambda^{-2\alpha}\sum_{\substack{l\\l\neq j}}\left(1+\frac{|\xi_j-\xi_l|}{\lambda^{-1+\alpha}}\right)^{-n}.
\end{equation}
\\
Since there is a main region in which the integral is large, and then a decay in its size in the surrounding regions, we use a dyadic decomposition of the unit circle to find the upper bound. This will take into account the different contributions from the integrals as $\frac{\lambda^{-1+\alpha}}{|\xi_j-\xi_l|}$ changes size.
\begin{lem}\label{dyadic}
For the set of $N=\lambda\gamma$ equally distributed $\xi_l$ on the unit circle, where $\xi_j$ is fixed, $$\sum_{\substack{l\\l\neq j}}\left(1+\frac{|\xi_j-\xi_l|}{\lambda^{-1+\alpha}}\right)^{-A}\leq \tilde C_{A}\gamma\lambda^\alpha$$ so long as $A\geq2$.
\end{lem}
\begin{proof}
By splitting the unit circle, in which the direction vectors are contained, into dyadic regions, the sum over the $\xi_l$ can be turned into a geometric sum. The first region is the region where $\frac{\lambda^{-1+\alpha}}{|\xi_j-\xi_l|}\geq1$ (where integration by parts does not work to give the bound as the contributions are large). This region is a sector of the unit circle (which contains all the direction vectors within this sector), which is symmetrical about the direction vector $\xi_j$. We can calculate the relationship between the angle the sector spans (in one direction from $\xi_j$); $\theta$ and the length of the line connecting $\xi_j$ and $\xi_l$; $|\xi_j-\xi_l|$, from the cosine rule ($c^2=a^2+b^2-2ab \cos C$). Since the lengths of the direction vectors, $|\xi_j|,|\xi_l|$, are 1; 
$$|\xi_j-\xi_l|^2=1+1-2\cos\theta=2-2\cos\theta=4\sin^2\left(\frac{\theta}{2}\right)$$
\begin{equation}\label{angle}
|\xi_j-\xi_l|=2\sin\left(\frac{\theta}{2}\right).
\end{equation}
The angle is important since it determines how many direction vectors are in the regions, as they are spaced evenly around the circle. Since there are $N=\gamma\lambda$ direction vectors, their angular density is $\frac{\gamma\lambda}{2\pi}$. From \eqref{angle} $\theta=2\arcsin\left(\frac{|\xi_j-\xi_l|}{2}\right)$, which for purposes of simplicity can be overestimated by $\theta\leq2|\xi_j-\xi_l|$ since $2\arcsin\left(\frac{|\xi_j-\xi_l|}{2}\right)\leq2|\xi_j-\xi_l|$. This means that in this initial region where $|\xi_j-\xi_l|\leq\lambda^{-1+\alpha}$, $\theta\leq2\lambda^{-1+\alpha}$, and consequently there are $2\lambda^{-1+\alpha}\cdot\frac{\gamma\lambda}{2\pi}=\frac{\gamma\lambda^\alpha}{\pi}$ direction vectors in the sectors on either side of $\xi_j$, meaning there are $\frac{2\gamma\lambda^\alpha}{\pi}$ direction vectors in the first region. This overestimates the number of direction vectors in this region, however since all the terms in the sum are positive this is acceptable for finding an upper bound.\\\\
The following regions are created by doubling the allowed sized of $|\xi_j-\xi_l|$, meaning the regions are characterised by the sets $X_\beta$, where each $X_\beta$ contains the $\xi_l$ which satisfy:
$$2^{\beta-1}\lambda^{-1+\alpha}\leq|\xi_j-\xi_l|< 2^\beta\lambda^{-1+\alpha}.$$
The sum can be changed to a sum involving $\beta$ as an index, but it needs a maximum value of $\beta$. We will denote this by $B$. Since $|\xi_j-\xi_l|\leq2$:
$$2\leq2^B\lambda^{-1+\alpha}$$
$$\log(\lambda^{1-\alpha})\leq(B-1)\log2$$
$$\frac{\log(\lambda^{1-\alpha})}{\log2}+1\leq B.$$
Therefore, we pick $B=\lceil{\frac{\log(\lambda^{1-\alpha})}{\log2}+1}\rceil$ since $B$ must be a natural number, and overestimating it will only include repeated terms in the sum, which is alright for an upper bound, since the terms are all positive.
The sum can now be rewritten as: 
$$\sum_{\substack{l\\l\neq j}}\left(1+\frac{|\xi_j-\xi_l|}{\lambda^{-1+\alpha}}\right)^{-A}\leq \sum_{\beta=0}^B\sum_{\substack{l\\{\xi_l\in X_\beta}}}\left(1+\frac{|\xi_j-\xi_l|}{\lambda^{-1+\alpha}}\right)^{-A}.$$
Since the $\beta=0$ term is the main term of the sum, it is helpful to separate this from the others:
$$\sum_{\substack{l\\l\neq j}}\left(1+\frac{|\xi_j-\xi_l|}{\lambda^{-1+\alpha}}\right)^{-A}\leq \sum_{\substack{l\\|\xi_j-\xi_l|\leq\lambda^{-1+\alpha}}}\left(1+\frac{|\xi_j-\xi_l|}{\lambda^{-1+\alpha}}\right)^{-A}+\sum_{\beta=1}^B\sum_{\substack{l\\\xi_l\in X_\beta}}\left(1+\frac{|\xi_j-\xi_l|}{\lambda^{-1+\alpha}}\right)^{-A}.$$
For the $\beta=0$ case the sum over $l$ is$\frac{2\gamma\lambda^\alpha}{\pi}$ and since this is the case where $\frac{|\xi_j-\xi_l|}{\lambda^{-1+\alpha}}$ is small, compared to 1, and can be ignored in the expression $\left(1+\frac{|\xi_j-\xi_l|}{\lambda^{-1+\alpha}}\right)$, the first term becomes:
$$\sum_{\substack{l\\|\xi_j-\xi_l|\leq\lambda^{-1+\alpha}}}\left(1+\frac{|\xi_j-\xi_l|}{\lambda^{-1+\alpha}}\right)^{-A}\lesssim\frac{2\gamma\lambda^\alpha}{\pi}\cdot(1)^{-A}=\hat{C_0}\gamma\lambda^{\alpha}.$$
For the following terms, we use the same way of estimating the number of direction vectors in each sector, as for the first region: $\theta\leq2|\xi_j-\xi_l|$. For the region with the outer boundary at $|\xi_j-\xi_l|=2^\beta\lambda^{-1+\alpha}$, this means that the boundary angle satisfies: $\theta\leq2^{\beta+1}\lambda^{-1+\alpha}$. As in the first region, this is overestimating the angle. The total number of direction vectors in that sector is $\frac{\lambda\gamma}{2\pi}\cdot2\cdot2^{\beta+1}\lambda^{-1+\alpha}=\frac{2^{\beta+1}\gamma\lambda^{\alpha}}{\pi}$ where the factor of two accounts for the angle going in both directions. (This counts all the direction vectors up to the boundary in each term, repeating the previous sections' ones which is not a problem for an upper bound). This evaluates the sum over $l$ for each value of $\beta$. 
Through overestimation of the $\left(1-\frac{|\xi_j-\xi_l|}{\lambda^{-1+\alpha}}\right)$ term, based on which set $X_\beta$ the $\xi_l$ is in, the second term becomes:
$$\sum_{\beta=1}^B\sum_{\substack{l\\\xi_l\in X_\beta}}\left(1+\frac{|\xi_j-\xi_l|}{\lambda^{-1+\alpha}}\right)^{-A}\leq\sum_{\beta=1}^B\sum_{l}\left(1+2^{\beta-1}\right)^{-A}.$$
Evaluating the sum over $l$ for each $\beta$ gives
$$\sum_{\beta=1}^B\sum_{\substack{l\\\xi_l\in X_\beta}}\left(1+\frac{|\xi_j-\xi_l|}{\lambda^{-1+\alpha}}\right)^{-A}\lesssim\frac{\gamma\lambda^\alpha}{\pi}\sum_{\beta=1}^B2^{\beta+1}\left(1+2^{\beta-1}\right)^{-A}.$$
Since $\beta$ is bigger than 1, the $2^{\beta-1}$ term will be dominant compared to 1 in $(1+2^{\beta-1})$:
$$ \sum_{\beta=1}^B\sum_{\substack{l\\\xi_l\in X_\beta}}\left(1+\frac{|\xi_j-\xi_l|}{\lambda^{-1+\alpha}}\right)^{-A}\lesssim\frac{\gamma\lambda^{\alpha}}{\pi}\sum_{\beta=1}^B2^{A+1}2^{\beta(1-A)}=\frac{2^{A+1}\gamma\lambda^{\alpha}}{\pi}\sum_{\beta=1}^B2^{\beta(1-A)}.$$
Since $A\geq2$ the geometric sum has a ratio ($2^{1-A}$) which is less than 1, so the series converges and the sum is bounded above. The sum starts at $\beta=1$ so the infinite sum converges to $\frac{r}{1-r}$ and $\sum_{\beta=1}^B2^{\beta(1-A)}$ is bounded above by $\frac{2^{1-A}}{1-2^{1-A}}$.
$$\sum_{\beta=1}^B\sum_{\substack{l\\\xi_l\in X_\beta}}\left(1+\frac{|\xi_j-\xi_l|}{\lambda^{-1+\alpha}}\right)^{-A}\leq\frac{2^{A+1}2^{1-A}\gamma\lambda^{\alpha}}{(1-2^{1-A})\pi}=\hat{C}\gamma\lambda^{\alpha}.$$

Therefore, adding the $\beta=0$ term and the other terms together we get:
$$\sum_{\substack{l\\l\neq j}}\left(1+\frac{|\xi_j-\xi_l|}{\lambda^{-1+\alpha}}\right)^{-A}\lesssim\hat{C_0}\gamma\lambda^\alpha+\hat{C}\gamma\lambda^\alpha=\tilde C_A\gamma\lambda^{\alpha}.$$
\end{proof}

 We can now return to the expression for the sums in the expectation, \eqref{predyadic}, and use the above result to estimate the contribution from the off-diagonal terms to the expectation value.
\\From the lemma above and \eqref{predyadic}: 
$$\sum_{\substack{l\\l\neq j}}\int{a^2(\lambda^\alpha |x|)e^{i\lambda x\cdot(\xi_j-\xi_l)}dx}\lesssim C\lambda^{-2\alpha}\cdot \tilde C_n\gamma\lambda^\alpha= K\gamma\lambda^{-\alpha}$$ (as long as the chosen fixed $n$ satisfies $n\geq2$).\\
Since this upper bound was not dependent on the $\xi_j$ that was fixed, it will hold for all $\xi_j$, and hence the sum over $j$ can be evaluated by multiplying by $N=\gamma\lambda$, giving
$$\sum_{\substack{j,l\\j\neq l}}\int{a^2(\lambda^\alpha |x|)e^{i\lambda x\cdot(\xi_j-\xi_l)}dx}\lesssim K\gamma^2\lambda^{1-\alpha}.$$
Substituting this bound for the double sum into \eqref{OintE}, the expression for the expectation gives us the final upper bound for the expectation value:
\begin{equation}\label{finalupperE}
\E(\lVert a_\lambda u\rVert^2)\leq4\pi\gamma\lambda^{1-2\alpha}+K(2p-1)^2\gamma^2\lambda^{1-\alpha}.
\end{equation}
\\
\\
\underline{Equidistribution:}
When the coefficients of the random wave are determined by a fair coin (i.e the probability of $C_j=+1$ is 0.5 and is equal to the probability of $C_j=-1$) the expectation has the same size as the volume of the region (once it has been normalised). This property of equi-distribution is interesting, and so we look for probabilities, $p$, where this property holds. Looking at \eqref{finalupperE}, this property holds if the two terms are the same size (since the first term is the expectation value for $p=0.5$). In this case $\gamma\simeq1$ so the terms are the same size when:
$$(2p-1)^2\lambda^{1-\alpha}=\mathcal{O}(\lambda^{1-2\alpha})$$
$$(2p-1)^2=\mathcal{O}(\lambda^{-\alpha})$$
$$2p-1=\mathcal{O}(\lambda^{-\frac{\alpha}{2}})$$
$$p=0.5+\mathcal{O}(\lambda^{-\frac{\alpha}{2}}).$$
This means that, up to constants, if the probability is $\lambda^{-\frac{\alpha}{2}}$ close to 0.5, the expectation will have the same size as the volume of the region. This is summarised in the following corollary.
\begin{cor}
A random wave given by \eqref{rw} where the coefficients are determined by an unfair coin ($C_j=+1$ has probability p and $C_j=-1$ has probability 1-p), and where $\gamma\simeq1$, has the property that $\E(\lVert a_\lambda u\rVert^2)\leq C\Vol(B_{\lambda^{-\alpha}}(0))$ if $$\lvert p-0.5\rvert\lesssim\lambda^{-\frac{\alpha}{2}}.$$
\end{cor}

\section{Proof of Theorem \ref{thm:mainexpint}}\label{sec:mainexpint}
In this section we use a different approach to get both a lower bound and an upper bound for the expectation. Each sum we are trying to evaluate has terms which come from a single function, which means that approximating the sums by integrals becomes a likely way to get bounds on the sum. Therefore to find bounds for the expectation, we use a Darboux integral based approach where we assume that $\gamma$ (the parameter which controls the number of direction vectors) is large, possibly tending to infinity.\\
From \eqref{Expectation} the sum, in the expression for the expectation, which needs to be evaluated is 
$$S=\sum_{j,l}\int{a^{2}(\lambda^\alpha}|x|)e^{i\lambda x\cdot(\xi_j-\xi_l)dx}=\int{a^{2}(\lambda^\alpha |x|)\sum_je^{i\lambda x\cdot\xi_j}\sum_l e^{-i\lambda x\cdot\xi_l}dx}.$$
The sum over $\xi_l$ can be parametrised using $\theta_l$ as the angle between $x$ and $\xi_l$, where $x$ is fixed. $\theta_l$ is chosen so that $\theta=0$ coincides with $\xi_l=\frac{x}{|x|}$ and $\theta_l\in[-\pi,\pi)$.
Similarly, the sum over $\xi_j$ can be parametrised using $\theta_j$ as the angle between $\xi_j$ and $x$ where $x$ is fixed. $\theta_j$ is chosen so that $\theta=0$ coincides with $\xi_j=\frac{x}{|x|}$ and $\theta_j\in[-\pi,\pi)$. Since the direction vectors are equally spaced around the unit circle, the width of the interval between two consecutive $\theta_l$ or $\theta_j$ is $\frac{2\pi}{\gamma\lambda}$.
\begin{equation}\label{sum}
S=\int{a^{2}(\lambda^\alpha |x|)\sum_je^{i\lambda |x|\cos\theta_j}\sum_l e^{-i\lambda |x|\cos\theta_l}dx}.
\end{equation}
We now want to turn these sums over $j$, $l$ into integrals over $\theta$. The following lemma gives us the ability to do so.
\begin{lem}\label{sumtoint}
For sums of the form $\sum_lf(\theta_l)$, where the $\theta_l$ are evenly spaced with a spacing of $\frac{2\pi}{\gamma\lambda}$ and $\theta_l\in[-\pi,\pi)$, and where $f(\theta)$ is a continuous function which satisfies $|f'(\theta)|\lesssim\lambda^{1-\alpha}$;
$$\frac{2\pi}{\gamma\lambda}\sum_lf(\theta_l)=\int_{-\pi}^\pi f(\theta)d\theta \text{ as }\gamma\to\infty.$$
\end{lem}
\begin{proof}

To be able to turn the sum into a Darboux integral, let $P=\{-\pi,\pi\}\cup\{\text{set of }\theta_l\}$ be a partition. Due to the even spacing of the $\theta_l$, the distance between the $\theta_1$ and $-\pi$ or $\theta_N$ and $\pi$ is less than $\frac{2\pi}{\gamma\lambda}$ as otherwise there would be another $\theta_l$ in between. We use the standard notation for Darboux sums: $m_i=\inf\{f(\theta)|\theta\in[\theta_i,\theta_{i+1}]\}$, $M_i=\sup\{f(\theta)|\theta\in[\theta_i,\theta_{i+1}]\}$, $L(f)=\sum_i\Delta\theta_im_i$ and $U(f)=\sum_i\Delta\theta_iM_i$.
Since the function $f(\theta)$ is continuous, it will attain its maximum and minimum on the interval $[\theta_i,\theta_{i+1}]$. To calculate $m_i$ and $M_i$ for each interval, we use a linear Taylor approximation about $\theta_i$ on each interval, where $\hat{\theta}\in[\theta_i,\theta_{i+1}]$: 
$$f(\theta)=f(\theta_i)+f'(\hat{\theta})(\theta-\theta_i).$$
In each interval we have $\theta-\theta_i\leq\frac{2\pi}{\gamma\lambda}$. This means that on the interval $[\theta_i,\theta_{i+1}]$,  we can write: $f(\theta)=f(\theta_i)+\mathcal{O}(\gamma^{-1}\lambda^{-\alpha})$. As this is true for any $\theta$ in the interval, it will be true for the values of $\theta$ which gives the maximum and minimum values of $f$: $m_i=f(\theta_i)+\mathcal{O}(\gamma^{-1}\lambda^{-\alpha})$ and $M_i=f(\theta_i)+\mathcal{O}(\gamma^{-1}\lambda^{-\alpha})$. Therefore $M_i-m_i=\mathcal{O}(\gamma^{-1}\lambda^{-\alpha})$. This is true for all $N+1$ intervals, meaning that $$U(f)-L(f)=\sum_i (M_i-m_i)\Delta\theta_i\lesssim\sum_i\mathcal{O}(\gamma^{-1}\lambda^{-\alpha})\frac{2\pi}{\gamma\lambda}=(N+1)\mathcal{O}(\gamma^{-2}\lambda^{-1-\alpha})=\mathcal{O}(\gamma^{-1}\lambda^{-\alpha})+\mathcal{O}(\gamma^{-2}\lambda^{-1-\alpha}).$$
Since $\gamma^{-1}\lambda^{-\alpha}>\gamma^{-2}\lambda^{-1-\alpha}$, the dominant error term is $\mathcal{O}(\gamma^{-1}\lambda^{-\alpha})$:
$$0\leq U(f)-L(f)\lesssim\mathcal{O}(\gamma^{-1}\lambda^{-\alpha}).$$
When $\gamma\rightarrow\infty$ this tends to zero, and hence the upper and lower Darboux sums are equal to each other in the limit.\\
Now for any partition
$$L(f)\leq\int_{-\pi}^{\pi}f(\theta)d\theta\leq U(f)$$
so by the squeeze theorem, the upper and lower Darboux sums are equal to the Darboux integral in the limit as $\gamma$ gets large.\\
We will now see that  $L(f)\leq C_{\lambda\gamma}\sum_lf(\theta_l)\leq U(f)$ (and in fact calculate $C_{\lambda\gamma}$), so that we can apply the squeeze theorem. On the intervals of the form $[\theta_l,\theta_{l+1}]$ (of which there are $N-1$), $M_l\geq f(\theta_l)$ and $m_l\leq f(\theta_l)$ . For these intervals $\Delta\theta_l=\frac{2\pi}{\lambda\gamma}$. On the interval $[\theta_N,\pi]$ $M_N\geq f(\theta_N)$ and $m_N\leq f(\theta_N)$. For this interval the width is $\Delta\theta_N\simeq k_N\frac{2\pi}{\lambda\gamma}$. On the interval $[-\pi, \theta_1]$, using a Taylor expansion about the point $\theta_1$, $M_0\geq f(\theta_1)$, and similarly $m_0\leq f(\theta_1)$. The width of this interval is $\Delta\theta_0\simeq k_1\frac{2\pi}{\gamma\lambda}$.\\
Therefore:
$$U(f)\geq\sum_{j=1}^{N-1}f(\theta_j)\cdot\frac{2\pi}{\gamma\lambda}+f(\theta_1)k_1\frac{2\pi}{\gamma\lambda}+f(\theta_N)k_N\frac{2\pi}{\gamma\lambda}=\frac{2\pi}{\gamma\lambda}\sum_{j=1}^Nf(\theta_j)+\mathcal{O}(\gamma^{-1}\lambda^{-1})$$
and
$$L(f)\leq\sum_{j=1}^{N-1}f(\theta_j)\cdot\frac{2\pi}{\gamma\lambda}+f(\theta_1)k_1\frac{2\pi}{\gamma\lambda}+f(\theta_N)k_N\frac{2\pi}{\gamma\lambda}=\frac{2\pi}{\gamma\lambda}\sum_{j=1}^Nf(\theta_j)+\mathcal{O}(\gamma^{-1}\lambda^{-1})$$
so
$$L(f)\leq\frac{2\pi}{\gamma\lambda}\sum_{j=1}^Nf(\theta_j)+\mathcal{O}(\gamma^{-1}\lambda^{-1})\leq U(f).$$
Now, let $\epsilon>0$; then we can pick $\Gamma$ so that if $\gamma\geq\Gamma$, 
$\left\lvert\mathcal{O}(\gamma^{-1}\lambda^{-1})\right\rvert\leq\epsilon.$
This is possible as the error term is converging as $\gamma^{-1}$. From this; 
$$L(f)\leq\frac{2\pi}{\gamma\lambda}\sum_{l=1}^Nf(\theta_l)\pm\epsilon\leq U(f)$$
$$L(f)-\epsilon\leq\frac{2\pi}{\gamma\lambda}\sum_{l=1}^Nf(\theta_l)\leq U(f)+\epsilon.$$
Since this is true for all $\epsilon>0$, it follows that 
$$L(f)\leq\frac{2\pi}{\gamma\lambda}\sum_{l=1}^\infty f(\theta_l)\leq U(f).$$
Taking the limit as $\gamma\rightarrow\infty$, using the squeeze theorem; 
$$\lim_{\gamma\to\infty}\frac{2\pi}{\gamma\lambda}\sum_lf(\theta_l)=\int_{-\pi}^\pi f(\theta)d\theta.$$
\end{proof}
To use this lemma to evaluate the sums in the expression for the expectation, \eqref{sum}, we need to show the bound on the derivative of the functions hold. In this case $f(\theta)=e^{\pm i\lambda|x|\cos\theta}$. Due to the size of the region of integration $|x|\lesssim\lambda^{-\alpha}$. This means:
$$f'(\theta)=\pm i\lambda|x|\sin\theta\cdot e^{\pm i\lambda |x|\cos\theta}$$ 
$$\lvert f'(\theta)\rvert=\left\lvert \pm i\lambda|x|\sin\theta\cdot e^{\pm i\lambda |x|\cos\theta}\right\rvert\leq \lambda|x|\lesssim \lambda^{1-\alpha}.$$

Therefore we use Lemma \ref{sumtoint} on the sums in \eqref{sum} to obtain: 
\begin{equation}\label{intl}
\lim_{\gamma\to\infty}\frac{2\pi}{\gamma\lambda}\sum e^{-i\lambda|x|\cos\theta_l}=\int_{-\pi}^{\pi}{e^{-i\lambda|x|\cos\theta}d\theta}
\end{equation}
and 
\begin{equation}\label{intj}
\lim_{\gamma\to\infty}\frac{2\pi}{\gamma\lambda}\sum e^{i\lambda|x|\cos\theta_j}=\int_{-\pi}^{\pi}{e^{i\lambda|x|\cos\theta}d\theta}.
\end{equation}

We also want to obtain a rate for the convergence in $\gamma$. In particular we want to write
$$\int a^{2}(\lambda^\alpha |x|)\sum_{j,l}e^{i\lambda|x|(\cos(\theta_{j})-\cos(\theta_{l}))}dx=\frac{\gamma^{2}\lambda^{2}}{4\pi^{2}}\int a^{2}(\lambda^{\alpha}|x|)\left[\left(\int_{-\pi}^{\pi} e^{i\lambda|x|\cos\theta}d\theta\right)\left(\int_{-\pi}^{\pi}e^{-i\lambda|x|\cos\psi}d\psi\right)\right]dx+E_{\gamma}$$
and obtain bounds for $|E_{\gamma}|$. If we write
\begin{align*}
\sum_{j}e^{i\lambda|x|\cos(\theta_{j})}&=I_{1}(x)+E_{1}(x)\\
\sum_{l}e^{-i\lambda|x|\cos(\theta_{l})}&=I_{2}(x)+E_{2}(x)\end{align*}
where $I_{1}(x)$ and $I_{2}(x)$ represent the integrals and $E_{1}(x),E_{2}(x)$ the errors. 
\begin{multline}\notag
\left|\int a^{2}(\lambda^\alpha |x|)\left(\sum_{j,l}e^{i\lambda|x|(\cos(\theta_{j})-\cos(\theta_{l}))}-I_{1}(x)I_{2}(x)\right)dx\right|\\
=\left|\int a^{2}(\lambda^\alpha |x|)\left(I_{1}(x)E_{2}(x)+I_{2}(x)E_{1}(x)+E_{1}(x)E_{2}(x)\right)dx\right|\\
\leq \norm{a_{\lambda}I_{1}}_{L^{2}}\norm{a_{\lambda}E_{2}}_{L^{2}}+\norm{a_{\lambda}I_{2}}_{L^{2}}\norm{a_{\lambda}E_{1}}_{L^{2}}+\norm{a_{\lambda}E_{1}}_{L^{2}}\norm{a_{\lambda}E_{2}}_{L^{2}}
\end{multline}
where we have applied Cauchy-Schwarz to obtain the last line.

So we need to obtain control on $\norm{I_{1}}_{L^{2}}$, $\norm{I_{2}}_{L^{2}}$, $\norm{E_{1}}_{L^{2}}$ and $\norm{E_{2}}_{L^{2}}$. The control on the $L^{2}$ norms coming from $I_{1}(x)$ and $I_{2}(x)$ will follow from the stationary phase computation we use to compute the $I_{1}(x)I_{2}(x)$ term. That just leaves the error terms. We can estimate them using much the same argument as we developed in Section \ref{sec:expsum} in Theorem \ref{OscillatoryInt} and Lemma \ref{dyadic}.

\begin{lem}\label{lem:Exerror}
Suppose
\begin{equation}E_{1}(x)=\sum_{j}e^{i\lambda|x|\cos(\theta_{j})}-\frac{\gamma\lambda}{2\pi}\int_{-\pi}^{\pi}e^{i\lambda|x|\cos(\theta)}d\theta=\sum_{j}e^{i\lambda|x|\cos(\theta_{j})}-\frac{\gamma\lambda}{2\pi}\int_{\S}e^{i\lambda x\cdot \xi}d\mu(\xi)\label{E1def}\end{equation}
\begin{equation}E_{2}(x)=\sum_{l}e^{-i\lambda|x|\cos(\theta_{l})}-\frac{\gamma\lambda}{2\pi}\int_{-\pi}^{\pi}e^{-i\lambda|x|\cos(\psi)}d\psi=\sum_{l}e^{i\lambda|x|\cos(\theta_{l})}-\frac{\gamma\lambda}{2\pi}\int_{\S}e^{i\lambda x\cdot \eta}d\mu(\eta)\label{E2def}\end{equation}
then
\begin{align}
\norm{a_{\lambda}E_{1}}_{L^{2}}&\lesssim\gamma\lambda^{\frac{1}{2}-\frac{\alpha}{2}}\label{E1est}\\
\norm{a_{\lambda}E_{2}}_{L^{2}}&\lesssim\gamma\lambda^{\frac{1}{2}-\frac{\alpha}{2}}\label{E2est}.\end{align}
\end{lem}

\begin{proof}
We will present the proof for $E_{1}$ (the proof for $E_{2}$ is identical). For $j=1,\dots,N-1$ denote the  arc of $\S$ lying between $\xi_{j}$ and $\xi_{j+1}$ by $\S_{j}$ and  let $\S_{N}$ be the arc between $\xi_{N}$ and $\xi_{1}$. We write
$$E_{1}(x)=\frac{\gamma\lambda}{2\pi}\sum_{j}\int_{\S_{j}}e^{i\lambda x\cdot \xi_{j}}d\mu(\xi)-\sum_{j}\int_{\S_{j}}e^{i\lambda x\cdot \xi}d\mu(\xi).$$
 Note that as we saw in the proof of Lemma \ref{sumtoint} a Taylor expansion of the exponential in $\xi$ around $\xi_{j}$ would give an estimate of
$$|E_{1}(x)|\lesssim\lambda^{1-\alpha}.$$
However, by exploiting the oscillatory nature of the $x$ integrals, we are able to improve on this. Expanding $|E_{1}(x)|^{2}$ we have that
\begin{align*}
\int a^{2}(\lambda^\alpha |x|)|E_{1}(x)|^{2}dx &=\left(\frac{\gamma\lambda}{2\pi}\right)^{2}\sum_{j,l}\Bigg(\int_{\S_{j}}\int_{\S_{l}}\int a^{2}(\lambda^\alpha |x|)e^{i\lambda x\cdot (\xi_{j}-\xi_{l})}dx d\mu(\xi)d\mu(\eta) \\
&-\int_{\S_{j}}\int_{\S_{l}}\int a^{2}(\lambda^\alpha |x|)e^{i\lambda x\cdot(\xi_{j}-\eta)}dx d\mu(\xi) d\mu(\eta)\\
&-\int_{\S_{j}}\int_{\S_{l}}\int a^{2}(\lambda^\alpha |x|)e^{i\lambda x\cdot(\xi-\xi_{l})}dx d\mu(\xi)d\mu(\eta)\\
&+\int_{\S_{j}}\int_{\S_{l}}\int a^{2}(\lambda^\alpha |x|)e^{i\lambda x\cdot (\xi-\eta)}dx d\mu(\xi)d\mu(\eta)\Bigg).\end{align*}
Now we can apply the integration by parts arguments of Theorem \ref{OscillatoryInt} to each term separately. Then using a Taylor expansion and the fact that $|\xi-\xi_{j}|\leq 2\pi\lambda^{-1}\gamma^{-1}$ and $|\eta-\xi_{l}|\leq 2\pi\lambda^{-1}\gamma^{-1}$ we obtain
$$\int a^{2}(\lambda^\alpha |x|)|E_{1}(x)|^{2}\leq C_{n}\frac{\lambda^{-2\alpha}}{\gamma}\sum_{j,l}\left(1+\frac{|\xi_{j}-\xi_{l}|}{\lambda^{-1+\alpha}}\right)^{-n}.$$
Finally we use the same dyadic decomposition of Lemma \ref{dyadic} to obtain
$$\int a^{2}(\lambda^\alpha |x|)|E_{1}(x)|^{2}\leq CN\lambda^{-\alpha}=C\gamma\lambda^{1-\alpha},$$
yielding the estimate
$$\norm{a_{\lambda}E_{1}}_{L^{2}}\lesssim \gamma^{\frac{1}{2}} \lambda^{\frac{1}{2}-\frac{\alpha}{2}}.$$
\end{proof}

We now compute $I_{1}(x)$ and $I_{2}(x)$. With these in hand we can compute
$$\int a^{2}(\lambda^{\alpha}|x|)I_{1}(x)I_{2}(x)dx$$
and estimate $\norm{I_{1}}_{L^{2}}$ and $\norm{I_{2}}_{L^{2}}$.
 We will do this by applying the method of stationary phase to the angular oscillatory integrals. We first consider the case $|x|>\lambda^{-1}$ since this allows us to look only at the leading terms in the expansions.

\begin{lem}\label{mosp}
If $|x|>\lambda^{-1}$ then 
$$\int_{-\pi}^\pi e^{\pm i\lambda|x|\cos\theta}d\theta\simeq2\sqrt{2\pi}(\lambda|x|)^{-\frac{1}{2}}\cos\left(\lambda|x|-\frac{\pi}{4}\right)+\mathcal{O}\left((\lambda|x|)^{-\frac{3}{2}}\right).$$
\end{lem}
\begin{proof}
This lemma is a standard result about Bessel functions using the asymptotic form, as 
$$\int_{-\pi}^\pi e^{\pm i\lambda|x|\cos\theta}d\theta=2\pi J_0(\lambda|x|)\simeq2\sqrt{2\pi}(\lambda|x|)^{-\frac{1}{2}}\cos\left(\lambda|x|-\frac{\pi}{4}\right)+\mathcal{O}\left((\lambda|x|)^{-\frac{3}{2}}\right).$$ However, we include an alternate proof using the method of stationary phase.\\
We will use the method of stationary phase outlined in the SEGwiki \cite{SEGwebsite} to approximate the integral. To avoid having to deal with boundary terms we introduce the smooth cutoff functions $b_1(\theta)$ which satisfies $b_1(\theta)=1$ when $\theta\in[-\frac{\pi}{4},\frac{\pi}{4}]$ and has compact support on $[-\frac{\pi}{2},\frac{\pi}{2}]$, and $b_2(\theta)=1-b_1(\theta)$. These cutoff functions allow us to rewrite the integral as:
\begin{align}\notag
\int_{-\pi}^{\pi}e^{\pm i\lambda|x|\cos\theta}d\theta&=\int_{-\pi}^{\pi}b_1(\theta)e^{\pm i\lambda|x|\cos\theta}d\theta+\int_{-\pi}^{\pi}b_2(\theta)e^{\pm i\lambda|x|\cos\theta}d\theta\\\label{cutoffint}
&=\int_{-\frac{\pi}{2}}^{\frac{\pi}{2}}b_1(\theta)e^{\pm i\lambda|x|\cos\theta}d\theta+\int_{-\pi}^{-\frac{\pi}{4}}b_2(\theta)e^{\pm i\lambda|x|\cos\theta}d\theta+\int_{\frac{\pi}{4}}^{\pi}b_2(\theta)e^{\pm i\lambda|x|\cos\theta}d\theta.
\end{align}
The stationary points are the points where the phase function $\phi(\theta)=\pm\cos\theta$ satisfies $\phi'(\theta)=\mp\sin\theta=0$, which in this case are $\theta=0,\pm\pi$. This means for the first integral in \eqref{cutoffint} the stationary point is an interior stationary point, for the second interval the stationary is at the lower endpoint of the integration and for the last integral the stationary point is at the upper endpoint of integration. There are three different formulas for these three cases, see \cite{SEGwebsite}: \\
\emph{Interior stationary point:} for a stationary point at $t=c$ where $a<c<b$, $$\int_a^bf(t)e^{i\lambda\phi(t)}dt\simeq e^{i\lambda\phi(c)+isgn(\phi''(c))\frac{\pi}{4}}f(c)\sqrt{\frac{2\pi}{\lambda|\phi''(c)|}}+\mathcal{O}(\lambda^{-\frac{3}{2}}).$$ In this case $\lambda=\lambda|x|$, $\phi(x)=\pm\cos\theta$, $f(t)=b_1(\theta)$, $a=-\frac{\pi}{2}$, $b=\frac{\pi}{2}$ and $c=0$. Noting that $b_1(0)=1$, this gives:
\begin{equation}\label{interior}
\int_{-\pi}^{\pi}b_1(\theta)e^{\pm i\lambda|x|\cos\theta}d\theta\simeq e^{\pm i\lambda|x|\mp i\frac{\pi}{4}}\sqrt{\frac{2\pi}{\lambda|x|}}+\mathcal{O}((\lambda|x|)^{-\frac{3}{2}}).
\end{equation}
\emph{Lower endpoint of integration:} for a stationary point at $t=a$,
$$\int_a^b{f(t)e^{i\lambda \phi(t)}dt}\simeq \frac{1}{2}e^{i\lambda\phi(a)+isgn(\phi''(a))\pi/4}\left\{f(a)\sqrt{\frac{2\pi}{\lambda|\phi''(a)|}}+\frac{2}{\lambda|\phi''(a)|}\left[f'(a)-\frac{\phi'''(a)}{3|\phi''(a)|}\right]e^{i\lambda sgn(\phi''(a))\pi/4}\right\}+\mathcal{O}\left(\lambda^{-\frac{3}{2}}\right).$$
In this case, $\lambda=\lambda|x|$, $\phi(x)=\pm\cos\theta$, $f(t)=b_2(\theta)$, $a=-\pi$ and $b=-\frac{\pi}{4}$. Noting that $b_2(-\pi)=1$, $b_2'(-\pi)=0$ and $\phi'''(a)=\pm\sin(-\pi)=0$ this gives:
\begin{equation}\label{lower}
\int_{-\pi}^{-\frac{\pi}{4}}b_2(\theta)e^{\pm i\lambda|x|\cos\theta}d\theta\simeq\frac{1}{2}e^{\mp i\lambda|x|\pm i\pi/4}\sqrt{\frac{2\pi}{\lambda|x|}}+\mathcal{O}((\lambda|x|)^{-\frac{3}{2}}).
\end{equation}
\emph{Upper endpoint of integration:} for a stationary point at $t=b$,
$$\int_a^b{f(t)e^{i\lambda \phi(t)}dt}\simeq \frac{1}{2}e^{i\lambda\phi(b)+isgn(\phi''(b))\pi/4}\left\{f(b)\sqrt{\frac{2\pi}{\lambda|\phi''(b)|}}-\frac{2}{\lambda|\phi''(b)|}\left[f'(b)-\frac{\phi'''(b)}{3|\phi''(b)|}\right]e^{i\lambda sgn(\phi''(b))\pi/4}\right\}+\mathcal{O}\left(\lambda^{-\frac{3}{2}}\right).$$
In this case, $\lambda=\lambda|x|$, $\phi(x)=\pm\cos\theta$, $f(t)=b_2(\theta)$, $a=\frac{\pi}{4}$ and $b=\pi$. Noting that $b_2(\pi)=1$, $b_2'(\pi)=0$ and $\phi'''(b)=\pm\sin(\pi)=0$ this gives:
\begin{equation}\label{upper}
\int_{\frac{\pi}{4}}^{\pi}b_2(\theta)e^{\pm i\lambda|x|\cos\theta}d\theta\simeq \frac{1}{2}e^{\mp i\lambda|x|\pm i\pi/4}\sqrt{\frac{2\pi}{\lambda|x|}}+\mathcal{O}((\lambda|x|)^{-\frac{3}{2}}).
\end{equation}
Putting \eqref{interior},\eqref{lower} and \eqref{upper} together gives us the overall integral:
\begin{align}\notag
\int_{-\pi}^{\pi}e^{\pm i\lambda|x|\cos\theta}d\theta&\simeq e^{\pm i\lambda|x|\mp i\frac{\pi}{4}}\sqrt{\frac{2\pi}{\lambda|x|}}+\frac{1}{2}e^{\mp i\lambda|x|\pm i\pi/4}\sqrt{\frac{2\pi}{\lambda|x|}}+\frac{1}{2}e^{\mp i\lambda|x|\pm i\pi/4}\sqrt{\frac{2\pi}{\lambda|x|}}+\mathcal{O}((\lambda|x|)^{-\frac{3}{2}})\\\notag
&=e^{\pm i\lambda|x|\mp i\frac{\pi}{4}}\sqrt{\frac{2\pi}{\lambda|x|}}+e^{\mp i\lambda|x|\pm i\pi/4}\sqrt{\frac{2\pi}{\lambda|x|}}+\mathcal{O}((\lambda|x|)^{-\frac{3}{2}})\\\notag
&=\sqrt{2\pi}\left(e^{\pm i\left(\lambda|x|-\frac{\pi}{4}\right)}+e^{\mp i\left(\lambda|x|-\frac{\pi}{4}\right)}\right)(\lambda|x|)^{-\frac{1}{2}}+\mathcal{O}\left((\lambda|x|)^{-\frac{3}{2}}\right)\\
\label{cosine}
&=2\sqrt{2\pi}(\lambda|x|)^{-\frac{1}{2}}\cos\left(\lambda|x|-\frac{\pi}{4}\right)+\mathcal{O}\left((\lambda|x|)^{-\frac{3}{2}}\right).
\end{align}
\end{proof}

This method of estimating the integral only works if $\lambda|x|>1$, as otherwise the later terms in the approximation will get very big. If this is not the case, and $\lambda|x|\leq1$ the function is not oscillating a lot, so 
$$\left\lvert\int_{-\pi}^{\pi}{e^{\pm i\lambda|x|\cos\theta}d\theta}\right\rvert\leq\int_{-\pi}^\pi\lvert e^{\pm i\lambda|x|\cos\theta}\rvert d\theta\leq\int_{-\pi}^\pi1d\theta= 2\pi$$
is an appropriate bound. From this we can see that when $\lambda|x|\leq1$; 
\begin{equation}\label{constant}
\int_{-\pi}^{\pi}{e^{\pm i\lambda|x|\cos\theta}d\theta}=\mathcal{O}(1).
\end{equation}
We can now compute
$$I=\int a^{2}(\lambda^{\alpha}|x|)I_{1}(x)I_{2}(x)dx,$$
 we can replace the integrals with the approximations from Lemma \eqref{mosp} and \eqref{constant}. Since the approximations for the interior integrals depend on the size of $\lambda|x|$, it is helpful to split the integral over $x$ into two regions: $|x|\leq\lambda^{-1}$ and $|x|>\lambda^{-1}$. The term $C\lambda^{-\frac{3}{2}}|x|^{-\frac{3}{2}}$ represents the $\mathcal{O}\left((\lambda|x|)^{-\frac{3}{2}}\right)$ terms.

$$I=\frac{\gamma^2\lambda^2}{4\pi^2}\left[\int_{|x|\leq\lambda^{-1}}a^2(\lambda^\alpha |x|)\mathcal{O}(1)dx+\int_{\lambda^{-1}\leq|x|}a^2(\lambda^\alpha |x|)\left(2\sqrt{2\pi}\lambda^{-\frac{1}{2}}|x|^{-\frac{1}{2}}\cos\left(\lambda|x|-\frac{\pi}{4}\right)+C\lambda^{-\frac{3}{2}}|x|^{-\frac{3}{2}}\right)^2dx\right]$$
\begin{multline}\notag
I=\frac{\gamma^2\lambda^2}{4\pi^2}\mathcal{O}\left(\int_{|x|\leq\lambda^{-1}}a^2(\lambda^\alpha |x|)dx\right)+\frac{\gamma^2\lambda^2}{4\pi^2}\int_{\lambda^{-1}\leq|x|}a^2(\lambda^\alpha |x|)8\pi\lambda^{-1}|x|^{-1}\cos^2\left(\lambda|x|-\frac{\pi}{4}\right)dx
\\+\mathcal{O}\left(\gamma^2\int_{\lambda^{-1}\leq|x|}a^2(\lambda^\alpha |x|)|x|^{-2}\cos\left(\lambda|x|-\frac{\pi}{4}\right)dx\right).
\end{multline}
Here the two error terms $\int C^2\lambda^{-3}|x|^{-3}dx$ and $\int C\lambda^{-2}|x|^{-2}\cos(\lambda|x|-\frac{\pi}{4})dx$ have been combined, so that we are only dealing with the leading error term.\\
In the region where $|x|\leq\lambda^{-1}$ the bump function $a^2(\lambda^\alpha |x|)=1$, so
\begin{multline}\notag
I=\frac{\gamma^2\lambda^2}{4\pi^2}\mathcal{O}\left(\int_{|x|\leq\lambda^{-1}}1dx\right)+\frac{2\gamma^2\lambda}{\pi}\int_{\lambda^{-1}\leq|x|}a^2(\lambda^\alpha |x|)|x|^{-1}\cos^2\left(\lambda|x|-\frac{\pi}{4}\right)dx\\+\mathcal{O}\left(\gamma^2\int_{\lambda^{-1}\leq|x|}a^2(\lambda^\alpha |x|)|x|^{-2}\cos\left(\lambda|x|-\frac{\pi}{4}\right)dx\right).
\end{multline}
Since the cutoff function is a radial function, as it is only a function of $|x|$ and not $x$, the second and third integrals can be converted into polar coordinates. We let $|x|=r$ and note that $dx=rdrd\theta$:
\begin{align}\notag
I&=\mathcal{O}(\gamma^2)+\frac{2\gamma^2\lambda}{\pi}\int_0^{2\pi}\int_{\lambda^{-1}}a^2(\lambda^\alpha r)r^{-1}\cos^2\left(\lambda r-\frac{\pi}{4}\right)rdrd\theta+\mathcal{O}\left(\gamma^2\int_0^{2\pi}\int_{\lambda^{-1}}^{\lambda^{-\alpha}}r^{-2}\cos\left(\lambda r-\frac{\pi}{4}\right)rdrd\theta\right)\\\notag
&=\mathcal{O}(\gamma^2)+4\gamma^2\lambda\int_{\lambda^{-1}}a^2(\lambda^\alpha r)\cos^2\left(\lambda r-\frac{\pi}{4}\right)dr+\mathcal{O}\left(\gamma^2\int_{\lambda^{-1}}^{\lambda^{-\alpha}}\frac{\cos\left(\lambda r-\frac{\pi}{4}\right)}{r}dr\right)\\\notag
&=\mathcal{O}(\gamma^2)+2\gamma^2\lambda\int_{\lambda^{-1}}a^2(\lambda^\alpha r)\left(1+\sin(2\lambda r)\right)dr+\mathcal{O}\left(\gamma^2\left[\frac{Si(\lambda r)+Ci(\lambda r)}{\sqrt2}\right]_{\lambda^{-1}}^{\lambda^{-\alpha}}\right)\\\notag
&=\mathcal{O}(\gamma^2)+2\gamma^2\lambda\int_{\lambda^{-1}}a^2(\lambda^\alpha r)\left(1+\sin(2\lambda r)\right)dr+\mathcal{O}\left(\gamma^2\left[Si(\lambda^{1-\alpha})+Ci(\lambda^{1-\alpha})-Si(1)-Ci(1)\right]\right),
\end{align}
where $Si(z)$ is the sine integral defined as $Si(z)=\int_0^z\frac{\sin t}{t}dt$ and $Ci(z)$ is the cosine integral defined as $Ci(z)=-\int_{\pi}^{\infty}\frac{\cos t}{t} dt$.
Provided that $\lambda^{1-\alpha}$ is large enough, which happens when $\lambda$ is large, i.e. when $\gamma\to\infty$, the functions $Si(\lambda^{1-\alpha})\to\frac{\pi}{2}$ and $Ci(\lambda^{1-\alpha})\to0$ do not grow, but tend to constants. As a result, the two error terms can be combined to give $\mathcal{O}(\gamma^2)$:
$$I=2\gamma^2\lambda\int_{\lambda^{-1}}a^2(\lambda^\alpha r)\left(1+\sin(2\lambda r)\right)dr+\mathcal{O}(\gamma^2).$$
Due to the support and other properties of the cutoff function, and since the term being multiplied by the cutoff function was squared and is hence positive, we can form the following bounds: 
\begin{equation}\label{alpha}
2\gamma^2\lambda\int_{\lambda^{-1}}^{\lambda^{-\alpha}}\left(1+\sin(2\lambda r)\right)dr+\mathcal{O}(\gamma^2)\leq I
\end{equation}
 and
\begin{equation}\label{2alpha}
I\leq2\gamma^2\lambda\int_{\lambda^{-1}}^{2\lambda^{-\alpha}}\left(1+\sin(2\lambda r)\right)dr+\mathcal{O}(\gamma^2).
\end{equation}
Dealing with the lower bound (\eqref{alpha}) first:
$$2\gamma^2\lambda\left[r-\frac{\cos(2\lambda r)}{2\lambda}\right]_{\lambda^{-1}}^{\lambda^{-\alpha}}+\mathcal{O}(\gamma^2)\leq I$$
$$2\gamma^2\lambda\left[\lambda^{-\alpha}-\frac{\cos(2\lambda^{1-\alpha})}{2\lambda}-\lambda^{-1}+\frac{\cos(2)}{2\lambda}\right]+\mathcal{O}(\gamma^2)\leq I$$
$$2\gamma^2\lambda^{1-\alpha}+\mathcal{O}(\gamma^2)\leq I.$$
We can pick $\lambda$ to be large enough, so that $\lvert\mathcal{O}(\gamma^2)\rvert\leq\gamma^2\lambda^{1-\alpha}$. As a result:
$$\gamma^2\lambda^{1-\alpha}\leq I.$$
Dealing with the upper bound (\eqref{2alpha}):
$$I\leq2\gamma^2\lambda\left[r-\frac{\cos(2\lambda r)}{2\lambda}\right]_{\lambda^{-1}}^{2\lambda^{-\alpha}}+\mathcal{O}(\gamma^2).$$
$$I\leq2\gamma^2\lambda\left[2\lambda^{-\alpha}-\frac{\cos(4\lambda^{1-\alpha})}{2\lambda}-\lambda^{-1}+\frac{\cos(2)}{2\lambda}\right]+\mathcal{O}(\gamma^2)$$
$$I\leq4\gamma^2\lambda^{1-\alpha}+\mathcal{O}(\gamma^2).$$
We can pick $\lambda$ to be large enough, so that $\mathcal{O}(\gamma^2)\leq\gamma^2\lambda^{1-\alpha}$. As a result:
$$I\leq5\gamma^2\lambda^{1-\alpha}.$$
Therefore:
\begin{equation}\gamma^2\lambda^{1-\alpha}\leq I\leq5\gamma^2\lambda^{1-\alpha}.\label{Iest}\end{equation}

Recall that Lemma \ref{lem:Exerror} gives us that
$$\norm{a_{\lambda}E_{1}}_{L^{2}}\lesssim \gamma^{\frac{1}{2}}\lambda^{\frac{1}{2}-\frac{\alpha}{2}}\quad\text{and}\quad \norm{a_{\lambda}E_{1}}_{L^{2}}\lesssim  \gamma^{\frac{1}{2}}\lambda^{\frac{1}{2}-\frac{\alpha}{2}}. $$
Since for fixed $x$, $I_{1}(x)$ and $I_{2}(x)$ enjoy the same upper bounds the upper bound for $I$ can be used to (upper) bound $\norm{a_{\lambda}I_{1}}_{L^{2}}^{2}$ and $\norm{a_{\lambda}I_{2}}_{L^{2}}^{2}$. Therefore
\begin{align*}
|S-I|&\leq \norm{a_{\lambda}I_{1}}_{L^{2}}\norm{a_\lambda E_{2}}_{L^{2}}+\norm{a_\lambda I_{2}}_{L^{2}}\norm{a_\lambda E_{1}}_{L^{2}}+\norm{a_\lambda E_{1}}_{L^{2}}\norm{a_\lambda E_{2}}_{L^{2}}\\
&\lesssim \gamma^{\frac{3}{2}}\lambda^{1-\alpha}+\gamma\lambda^{1-\alpha}.\end{align*}
Since we are only considering the case where $\gamma$ is large we can then sweep these errors into \eqref{Iest} to obtain a $c,C$ so that
$$c\gamma^2\lambda^{1-\alpha}\leq S \leq C\gamma^2\lambda^{1-\alpha}.$$

We have now obtained both a lower bound and an upper bound for the sum $S$. Since this appears in the expression for the expectation \eqref{Expectation}, as $\E(\lVert a_\lambda u\rVert^2)=N\int{a^2(\lambda^\alpha |x|)dx}+(2p-1)^2S$, we can substitute the bounds for $S$ to obtain an upper and lower bound for the expectation, in the case where $\gamma\to\infty$. We also use the property of the cutoff function to obtain the required bounds for the first integral which comes from the diagonal terms.
\begin{equation}\label{boundsE}
\pi\gamma\lambda^{1-2\alpha}+c(2p-1)^2\gamma^2\lambda^{1-\alpha}<\E(\lVert a_\lambda u\rVert^2)<4\pi\gamma\lambda^{1-2\alpha}+C(2p-1)^2\gamma^2\lambda^{1-\alpha}.
\end{equation}\\
\underline{Equidistribution:} 
\eqref{boundsE} gives the bounds on the expectation in the case where $\gamma\to\infty$. After normalisation, the second term in them is on the scale of $\gamma\lambda^{-\alpha}$. Therefore, the weak property of equidistribution, \eqref{equiXw}, holds when
 $$(2p-1)^2\gamma\lambda^{-\alpha}=\mathcal{O}(\lambda^{-2\alpha})$$
 $$(2p-1)^2=\mathcal{O}(\lambda^{-\alpha}\gamma^{-1})$$
 $$p=0.5+\mathcal{O}(\lambda^{-\frac{\alpha}{2}}\gamma^{-\frac{1}{2}}).$$
 Therefore, up to constants, if the probability is $\lambda^{-\frac{\alpha}{2}}\gamma^{-\frac{1}{2}}$ close to 0.5, the expectation will scale with the volume of the region, and the weak equidistribution property holds. This is summarised by the following corollary.
\begin{cor}\label{cor:equi}
A random wave given by \eqref{rw} where the coefficients are determined by an unfair coin ($C_j=+1$ has probability p and $C_j=-1$ has probability 1-p), and where $\gamma\to\infty$, satisfies the condition on the expectation for the weak property of equidistribution, given by \eqref{equiXw}, if $$\lvert p-0.5\rvert\lesssim\lambda^{-\frac{\alpha}{2}}\gamma^{-\frac{1}{2}}.$$
\end{cor}
\section{Proof of Theorem \ref{thm:mainvar}}\label{sec:mainvar}
The variance of a quantity is given by $\sigma^2(\lVert a_\lambda u\rVert^2)=\E((\lVert a_\lambda u\rVert^2-\E(\lVert a_\lambda u\rVert^2))^2)$.
We can use the independence of the coefficients, $C_j$ to obtain an expression for the variance in terms of the same off-diagonal terms involving oscillatory integrals, which were considered in the expectation. 
To simplify the expressions in the following calculations we define $I_{jl}=\int{a^2(\lambda^\alpha |x|)e^{i\lambda x\cdot(\xi_j-\xi_l)}dx}$. In the case where $j=l$ we have $I_{jl}=\int{a^2(\lambda^\alpha|x|)dx}$. It also follows that $|I_{jl}|=|I_{lj}|$.

\begin{prop}\label{varianceprop} 
The variance of $\lVert a_\lambda u\rVert^2$ for a random wave $u(x)$ given by \eqref{rw}, where the coefficients are determined by an unfair coin ($C_j=+1$ has probability p and $C_j=-1$ has probability 1-p), can be expressed as
\begin{multline}\notag
\sigma^2(\lVert a_\lambda u\rVert^2)=\left[1-(2p-1)^2\right]^2\left[\sum_{\substack{j,l\\j\neq l}}I_{jl}^2+\sum_{\substack{j,l\\j\neq l}}I_{jl}I_{lj}\right]+\left[(2p-1)^2-(2p-1)^4\right]\cdot\left[\sum_j\left(\sum_{l\neq j}I_{jl}\right)\left(\sum_{n\neq j}I_{jn}\right)\right.\\\left.+\sum_j\left(\sum_{l\neq j}I_{jl}\right)\left(\sum_{m\neq j}I_{mj}\right)+\sum_l\left(\sum_{j\neq l}I_{jl}\right)\left(\sum_{n\neq l}I_{ln}\right)+\sum_l\left(\sum_{j\neq l}I_{jl}\right)\left(\sum_{m\neq l}I_{ml}\right)\right].
\end{multline}
\end{prop}
\begin{proof}
We begin by substituting the formula for the random wave into the expression for the variance:
\begin{align}\notag
\sigma^2(\lVert a_\lambda u\rVert^2)&=\sum_kP_k\left[\iint\sum_{j,l,m,n}a^2(\lambda^\alpha |x|)a^2(\lambda^\alpha |y|)C^{(k)}_jC^{(k)}_lC^{(k)}_mC^{(k)}_ne^{i\lambda x\cdot(\xi_j-\xi_l)}e^{i\lambda y\cdot(\xi_m-\xi_n)}dxdy-\E(\lVert a_\lambda u\rVert^2)^2\right]\\
&=\sum_kP_k\left[\sum_{j,l,m,n}C^{(k)}_jC^{(k)}_lC^{(k)}_mC^{(k)}_nI_{jl}I_{mn}-\E(\lVert a_\lambda u\rVert^2)^2\right].\label{variance1}
\end{align}
From \eqref{Expectation} we have an expression for the expectation in terms of the integrals in the diagonal and off-diagonal terms: $$\E(\lVert a_\lambda u\rVert^2)=N\int{a^2(\lambda^\alpha |x|)dx}+(2p-1)^2\sum_{\substack{j,l\\j\neq l}}\int{a^2(\lambda^\alpha |x|)e^{i\lambda x\cdot(\xi_j-\xi_l)}dx}=N\int{a^2(\lambda^\alpha |x|)dx}+(2p-1)^2\sum_{\substack{j,l\\j\neq l}}I_{jl},$$ which can be substituted into the expression for the variance to achieve some cancellation. To see the cancellation we need to compute $\E^2$:
\begin{multline}\notag
\E(\lVert a_\lambda u\rVert^2)^2=N^2\iint{a^2(\lambda^\alpha |x|)a^2(\lambda^\alpha |y|)dxdy}+2N\int{a^2(\lambda^\alpha |y|)dy}\cdot(2p-1)^2\sum_{\substack{j,l\\j\neq l}}I_{jl}+(2p-1)^4\cdot\sum_{\substack{j,l\\j\neq l}}I_{jl}\cdot\sum_{\substack{m,n\\m\neq n}}I_{mn}
\end{multline}
\begin{multline}\label{E2}
\E(\lVert a_\lambda u\rVert^2)^2=N^2\iint{a^2(\lambda^\alpha |x|)a^2(\lambda^\alpha |y|)dxdy}+2N(2p-1)^2\sum_{\substack{j,l\\j\neq l}}I_{jl}\int{a^2(\lambda^\alpha |y|)dy}+(2p-1)^4\sum_{\substack{j,l,m,n\\j\neq l\\m\neq n}}I_{jl}I_{mn}.
\end{multline}
This expression can be taken outside of the sum over $k$ since it does not depend on $k$, and the resulting sum: $\sum_kP_k=1$. The first term in the expression for the variance, \eqref{variance1}, needs to be split into different combinations of $j$, $l$, $m$ and $n$ so that they can cancel with terms in the expression for $\E^2$, \eqref{E2};
$$\sum_kP_k\left(\sum_{j,l,m,n}C^{(k)}_jC^{(k)}_lC^{(k)}_mC^{(k)}_nI_{jl}I_{mn}\right).$$
Important terms are those in which there are pairs of $j$, $l$, $m$ or $n$ which are equal, since in those cases the coefficients aren't dependant on probabilities and in some cases where $j=l$ or $m=n$ the exponents simplify:\\
\emph{Case 1: $j=l$ and $m=n$:} 
$$\sum_kP_k\sum_{j,m}\iint{a^2(\lambda^\alpha |x|)a^2(\lambda^\alpha |y|)(C^{(k)}_j)^2(C^{(k)}_m)^2dxdy}=N^2\iint{a^2(\lambda^\alpha |x|)a^2(\lambda^\alpha |y|)dxdy}.$$
This cancels out with the first term in \eqref{E2}.\\
\emph{Case 2: $j=l$ but $m\neq n$:}
\begin{equation}\notag
\sum_kP_k\sum_{\substack{j,m,n\\m\neq n}}(C^{(k)}_j)^2C^{(k)}_mC^{(k)}_nI_{mn}\int{a^2(\lambda^\alpha |x|)dx}=N\sum_{\substack{m,n\\m\neq n}}\sum_kP_kC^{(k)}_mC^{(k)}_nI_{mn}\int{a^2(\lambda^\alpha |x|)dx}.
\end{equation}
From Lemma \ref{prob}, we know that $\sum_kP_kC^{(k)}_mC^{(k)}_n=(2p-1)^2$ and hence: 
\begin{equation}\notag
\sum_kP_k\sum_{\substack{j,m,n\\m\neq n}}(C^{(k)}_m)^2C^{(k)}_mC^{(k)}_nI_{mn}\int{a^2(\lambda^\alpha |x|)dx}=N(2p-1)^2\sum_{\substack{m,n\\m\neq n}}I_{mn}\int a^2(\lambda^\alpha |x|)dx.
\end{equation}
\emph{Case 3: $j\neq l$ but $m=n$:} \\
This case is similar to the $j=l$ but $m\neq n$ case, and so:
\begin{equation}\notag
\sum_kP_k\sum_{\substack{j,l,m\\j\neq l}}(C^{(k)}_j)^2C^{(k)}_jC^{(k)}_l I_{jl}\int{a^2(\lambda^\alpha |y|)dy}=N(2p-1)^2\sum_{\substack{j,l\\j\neq l}}I_{jl}\int a^2(\lambda^\alpha |y|)dy.
\end{equation}
Since the indices are arbitrary, these two cases, (Cases 2 and 3), cancel out the second term in the expression for $\E^2$, \eqref{E2}.
This leaves:
\begin{align}\notag
\sigma^2(\lVert a_\lambda u\rVert^2)&=\sum_kP_k\sum_{\substack{j,l,m,n\\j\neq l\\m\neq n}}C^{(k)}_jC^{(k)}_lC^{(k)}_mC^{(k)}_nI_{jl}I_{mn}-(2p-1)^4\sum_{\substack{j,l,m,n\\j\neq l\\m\neq n}}I_{jl}I_{mn}\\
&=\left[\sum_kP_kC^{(k)}_jC^{(k)}_lC^{(k)}_mC^{(k)}_n-(2p-1)^4\right]\sum_{\substack{j,l,m,n\\j\neq l\\m\neq n}}I_{jl}I_{mn}.\label{variance2}
\end{align}
For the terms where $j, l, n$ and $m$ are independent we need to evaluate $\sum_kP_kC^{(k)}_jC^{(k)}_lC^{(k)}_mC^{(k)}_n$ in terms of $p$. In this calculation we fix $j$, $l$, $m$ and $n$. Since the values of the entries of interest ($j$, $l$, $m$ and $n$) are independent of each other, the probabilities of the different combinations of coefficients can be calculated as follows:
$$\sum_kP_kC^{(k)}_jC^{(k)}_lC^{(k)}_mC^{(k)}_n=\E(C_j^{(k)}C_l^{(k)}C^{(k)}_mC^{(k)}_n)=\E(C_j^{(k)})\cdot\E(C_l^{(k)})\cdot\E(C_m^{(k)})\cdot\E(C_n^{(k)})=(2p-1)^4$$
by a similar argument as Lemma \ref{prob}. This is the same coefficient as that of the second term in \eqref{variance2}, so all these terms will cancel.
The remaining cases are the terms where $(j,l,m,n)$ are not all distinct but $j\neq l$ and $m\neq n$. These cases consist of the other pair terms,   ($j=m$, $l=n$) and ($j=n$, $l=m$), as well as the four cases where there is only one pair.\\
\emph{Case 4: $j=m$ and $l=n$ $(m\neq l)$:} contributes a term of the form:
\begin{equation}\notag
\left[1-(2p-1)^4\right]\sum_{\substack{j,l\\j\neq l}}I_{jl}^2.
\end{equation}
\emph{Case 5: $j=n$ and $l=m$ $(n\neq l)$:} contributes a term of the form:
\begin{equation}\notag
\left[1-(2p-1)^4\right]\sum_{\substack{j,l\\j\neq l}}I_{jl}I_{lj}.
\end{equation}
\emph{Case 6: $j=m$ and $l\neq n\neq j$:} contributes the following terms:
\begin{equation}\notag
\left[\sum_kP_kC^{(k)}_lC^{(k)}_n-(2p-1)^4\right]\sum_{\substack{j,l,n\\j\neq l, j\neq n\\l\neq n}}I_{jl}I_{jn}=\left[(2p-1)^2-(2p-1)^4\right]\left[\sum_j\left(\sum_{l\neq j}I_{jl}\right)\left(\sum_{n\neq j}I_{jn}\right)-\sum_{\substack{j,l\\j\neq l}}I_{jl}^2\right]
\end{equation}
as by Lemma \ref{prob} we know that $\sum_kP_kC^{(k)}_mC^{(k)}_n=(2p-1)^2$. The term $\sum_{\substack{j,l\\j\neq l}}I_{j,l}^2$ corresponds to subtracting the terms corresponding to $l=n$.\\
\emph{Case 7: $j=n$ and $l\neq m\neq j$, Case 8: $l=m$ and $j\neq n\neq l$} and \emph{Case 9: $l=n$ and $j\neq m\neq l$:} are similar to Case 6 and contribute the following, where we manually remove the terms where $l=m$, $j=n$ and $j=m$ for each case respectively:
\begin{multline}\notag
\left[(2p-1)^2-(2p-1)^4\right]\left[\sum_j\left(\sum_{l\neq j}I_{jl}\right)\left(\sum_{m\neq j}I_{mj}\right)-\sum_{\substack{j,l\\j\neq l}}I_{jl}I_{lj}\right]\\
+\left[(2p-1)^2-(2p-1)^4\right]\left[\sum_l\left(\sum_{j\neq l}I_{jl}\right)\left(\sum_{n\neq l}I_{ln}\right)-\sum_{\substack{j,l\\j\neq l}}I_{jl}I_{lj}\right]\\
+\left[(2p-1)^2-(2p-1)^4\right]\left[\sum_l\left(\sum_{j\neq l}I_{jl}\right)\left(\sum_{m\neq l}I_{ml}\right)-\sum_{\substack{j,l\\j\neq l}}I_{jl}^2\right].
\end{multline}
Combining all these terms gives the following expression for the variance:
\begin{multline}\notag
\sigma^2(\lVert a_\lambda u\rVert^2)=\left[1-(2p-1)^4\right]\left[\sum_{\substack{j,l\\j\neq l}}I_{jl}^2+\sum_{\substack{j,l\\j\neq l}}I_{jl}I_{lj}\right]+\left[(2p-1)^2-(2p-1)^4\right]\cdot\left[\sum_j\left(\sum_{l\neq j}I_{jl}\right)\left(\sum_{n\neq j}I_{jn}\right)\right.\\\left.+\sum_j\left(\sum_{l\neq j}I_{jl}\right)\left(\sum_{m\neq j}I_{mj}\right)+\sum_l\left(\sum_{j\neq l}I_{jl}\right)\left(\sum_{n\neq l}I_{ln}\right)+\sum_l\left(\sum_{j\neq l}I_{jl}\right)\left(\sum_{m\neq l}I_{ml}\right)-2\left(\sum_{\substack{j,l\\j\neq l}}I_{jl}^2+\sum_{\substack{j,l\\j\neq l}}I_{jl}I_{lj}\right)\right].
\end{multline}
These terms can be regrouped as follows:
\begin{multline}\notag
\sigma^2(\lVert a_\lambda u\rVert^2)=\left[1-2(2p-1)^2+(2p-1)^4\right]\left[\sum_{\substack{j,l\\j\neq l}}I_{jl}^2+\sum_{\substack{j,l\\j\neq l}}I_{jl}I_{lj}\right]+\left[(2p-1)^2-(2p-1)^4\right]\cdot\left[\sum_j\left(\sum_{l\neq j}I_{jl}\right)\left(\sum_{n\neq j}I_{jn}\right)\right.\\\left.+\sum_j\left(\sum_{l\neq j}I_{jl}\right)\left(\sum_{m\neq j}I_{mj}\right)+\sum_l\left(\sum_{j\neq l}I_{jl}\right)\left(\sum_{n\neq l}I_{ln}\right)+\sum_l\left(\sum_{j\neq l}I_{jl}\right)\left(\sum_{m\neq l}I_{ml}\right)\right].
\end{multline}
\end{proof}
From here we can continue the proof of Theorem \ref{thm:mainvar}, by computing an upper bound for the integrals in the above expression, using methods from section \ref{sec:expsum}.\\
\underline{Upper Bound:}
We can recognise the square $(1-(2p-1))^2=1-2(2p-1)^2+(2p-1)^4$ from the statement of Proposition \ref{varianceprop}, and using the triangle inequality over the different terms as well as the finite sums we get:
\begin{multline}\notag
|\sigma^2(\lVert a_\lambda u\rVert^2)|\leq\left[1-(2p-1)^2\right]^2\left[\sum_{\substack{j,l\\j\neq l}}|I_{jl}|^2+\sum_{\substack{j,l\\j\neq l}|}|I_{jl}||I_{lj}|\right]+\left[(2p-1)^2-(2p-1)^4\right]\cdot\left[\sum_j\left(\sum_{l\neq j}|I_{jl}|\right)\left(\sum_{n\neq j}|I_{jn}|\right)\right.\\\left.+\sum_j\left(\sum_{l\neq j}|I_{jl}|\right)\left(\sum_{m\neq j}|I_{mj}|\right)+\sum_l\left(\sum_{j\neq l}|I_{jl}|\right)\left(\sum_{n\neq l}|I_{ln}|\right)+\sum_l\left(\sum_{j\neq l}|I_{jl}|\right)\left(\sum_{m\neq l}|I_{ml}|\right)\right].
\end{multline}
The absolute values allow us to group the cases together, using $|I_{jl}|=|I_{lj}|$, and substituting in for $I_{jl}$ gives:
\begin{multline}\notag
\sigma^2\leq2\left[1-(2p-1)^2\right]^2\left[\sum_{\substack{j,l\\j\neq l}}\left\lvert\int{a^2(\lambda^\alpha |x|)e^{i\lambda x\cdot(\xi_j-\xi_l)}dx}\right\rvert^2\right]\\
+4\left[(2p-1)^2-(2p-1)^4\right]\cdot\left[\sum_j\left(\sum_{l\neq j}\left|\int a^2(\lambda^\alpha |x|) e^{i\lambda x\cdot(\xi_j-\xi_l)}dx\right|\right)^2\right].
\end{multline}
We now use the same upper bound,  \eqref{bothcases}, for the integrals in this expression, which was obtained as a result of Theorem \ref{OscillatoryInt} in the calculation for the expectation: 
\begin{multline}\notag
\sigma^2(\lVert a_\lambda u\rVert^2)\lesssim 2\left[1-(2p-1)^2\right]^2\left[\sum_{\substack{j,l\\j\neq l}}\left(\lambda^{-2\alpha}\left(1+\frac{|\xi_j-\xi_l|}{\lambda^{-1+\alpha}}\right)^{-n}\right)^2\right]\\
+4\left[(2p-1)^2-(2p-1)^4\right]\cdot\left[\sum_j\left(\sum_{l\neq j}\left(\lambda^{-2\alpha}\left(1+\frac{|\xi_j-\xi_l|}{\lambda^{-1+\alpha}}\right)^{-n}\right)\right)^2\right].
\end{multline}
\begin{multline}\notag
\sigma^2(\lVert a_\lambda u\rVert^2)\lesssim2\left[1-(2p-1)^2\right]^2\lambda^{-4\alpha}\sum_j\left[\sum_{\substack{l\\l\neq j}}\left(1+\frac{|\xi_j-\xi_l|}{\lambda^{-1+\alpha}}\right)^{-2n}\right]\\
+4\left[(2p-1)^2-(2p-1)^4\right]\lambda^{-4\alpha}\sum_j\left[\sum_{\substack{l\\l\neq j}}\left(1+\frac{|\xi_j-\xi_l|}{\lambda^{-1+\alpha}}\right)^{-n}\right]^2.
\end{multline}
We can use Lemma \ref{dyadic}, which evaluates the sum over $l$ using a dyadic decomposition, to rewrite this as: 
$$\sigma^2(\lVert a_\lambda u\rVert^2)\lesssim2\left[1-(2p-1)^2\right]^2\lambda^{-4\alpha}\sum_j\left[ \gamma\lambda^\alpha\right]+4\left[(2p-1)^2-(2p-1)^4\right]\lambda^{-4\alpha}\sum_j\left[ \gamma\lambda^\alpha\right]^2$$
where the implicit constant is determined by the number of integration by parts necessary to estimate the inner sum (in this case at least two iterations are necessary). Since this is independent of $j$, the sum over $j$ is evaluated by multiplying by $N=\gamma\lambda$. Simplifying this then gives the desired upper bound:
$$\sigma^2(\lVert a_\lambda u\rVert^2)\lesssim2\left[1-(2p-1)^2\right]^2\lambda^{-4\alpha}\cdot\lambda\gamma\cdot\gamma\lambda^\alpha+4\left[(2p-1)^2-(2p-1)^4\right]\lambda^{-4\alpha}\cdot\lambda\gamma\cdot\gamma^2\lambda^{2\alpha}$$
\begin{equation}\label{finalupperV}
\sigma^2(\lVert a_\lambda u\rVert^2)\lesssim\left[1-(2p-1)^2\right]^2\gamma^2\lambda^{1-3\alpha}+(2p-1)^2\left[1-(2p-1)^2\right]\gamma^3\lambda^{1-2\alpha}.
\end{equation}
\underline{Equidistribution:} 
\eqref{finalupperV} gives the upper bound on the variance in the case where $\gamma\to\infty$. After normalisation, dividing by $N^2=\gamma^2\lambda^2$, this becomes:
$$\sigma^2\lesssim\lambda^{-1-3\alpha}[1-(2p-1)^2]^{2}+\gamma\lambda^{-1-2\alpha}(2p-1)^2[1-(2p-1)^2]$$
If we assume the condition given by Corollary \ref{cor:equi}, i.e. $(2p-1)^2=\mathcal{O}(\lambda^{-\alpha}\gamma^{-1})$, the second term is of the same size as the first:
$$\sigma^2\lesssim\lambda^{-1-3\alpha}[1-(2p-1)^2]^2+\lambda^{-1-3\alpha}[1-(2p-1)^2].$$
Therefore, the requirement on the variance for equidistribution, \eqref{varXest}, holds when $\alpha<1$ and $(2p-1)^2=\mathcal{O}(\lambda^{-\alpha}\gamma^{-1})$. This is summarised by the following corollary.
\begin{cor}\label{cor:var}
A random wave given by \eqref{rw} where the coefficients are determined by an unfair coin ($C_j=+1$ has probability p and $C_j=-1$ has probability 1-p), and where $\gamma\to\infty$, satisfies the condition for equidistribution on the variance, given by \eqref{varXest}, and hence the weak equidistribution property given by \eqref{equiXw}, if $\lvert p-0.5\rvert\lesssim\lambda^{-\frac{\alpha}{2}}\gamma^{-\frac{1}{2}}$ and $\alpha<1$.
\end{cor}

\bibliographystyle{plain}
\bibliography{references}

\end{document}